\documentclass[11pt, reqno]{amsart}
\usepackage{ifthen}
\usepackage{latexsym}
\usepackage{graphicx}
\usepackage{amsmath}
\usepackage{amssymb}
\usepackage{amsfonts}
\usepackage{wasysym}
\usepackage{amsthm,amssymb,amscd}
\usepackage{epsfig}
\usepackage{epstopdf,fancyhdr}
\usepackage{extarrows}

\theoremstyle{plain}
\newtheorem{thm}{Theorem}[section]
\newtheorem{prop}[thm]{Proposition}
\newtheorem{lem}[thm]{Lemma}
\newtheorem{cor}[thm]{Corollary}

\theoremstyle{definition}
\newtheorem*{notation}{Notation}
\newtheorem{rem}[thm]{Remark}
\newtheorem{defn}{Definition}[section]
\newtheorem{eg}[thm]{Example}

\numberwithin{equation}{section}

\newcommand{\bthm}{\begin{thm}}
\newcommand{\ethm}{\end{thm}}
\newcommand{\bprop}{\begin{prop}}
\newcommand{\eprop}{\end{prop}}
\newcommand{\bcor}{\begin{cor}}
\newcommand{\ecor}{\end{cor}}
\newcommand{\blem}{\begin{lem}}
\newcommand{\elem}{\end{lem}}
\newcommand{\bca}{\begin{cases}}
\newcommand{\eca}{\end{cases}}
\newcommand{\brem}{\begin{rem}}
\newcommand{\erem}{\end{rem}}
\newcommand{\bpm}{\begin{pmatrix}}
\newcommand{\epm}{\end{pmatrix}}
\newcommand{\bdefn}{\begin{defn}}
\newcommand{\edefn}{\end{defn}}
\newcommand{\bsub}{\begin{subtitle}}
\newcommand{\esub}{\end{subtitle}}
\newcommand{\ben}{\begin{enumerate}}
\newcommand{\een}{\end{enumerate}}
\newcommand{\beg}{\begin{eg}}
\newcommand{\eeg}{\end{eg}}
\newcommand{\beq}{\begin{equation}}
\newcommand{\eeq}{\end{equation}}

\def\ms{\medskip}

\def\ni{\noindent}

\def\Re{{\rm Re\/}}

\def\diag{{\rm diag\/}}
\def\Span{{\rm Span\/}}

\def \a {\alpha}
\def \b {\beta}

\def \e {\epsilon}

\def \l {\lambda}
\def \L {\Lambda}

\def \s {\sigma}

\def \ud{\mathrm{d}}

\def\R{\mathbb{R} }
\def\C{\mathbb{C}}

\def\Z{\mathbb{Z}}

\def\fg{\mathfrak{g}}
\def \Tr{\mathrm{Tr}}

\def\cg{{\mathcal {G}}}

\DeclareMathOperator{\GL}{\mathrm{GL}}
\DeclareMathOperator{\SL}{\mathrm{SL}}

\DeclareMathOperator{\SO}{\mathrm{SO}}

\DeclareMathOperator{\sli}{\mathrm{sl}}

\begin{document}

\title{DRESSING ACTIONS ON PROPER DEFINITE AFFINE SPHERES}

\author{Zhicheng Lin, Gang Wang}
\address{Wuhan Institute of Physics and Mathematics, Chinese Academy of Sciences,
Wuhan, Hubei 430071, China}
\email{flyriverms@qq.com}

\address{Shandong University,
Jinan, Shandong 250100, China}
\email{wg110789@sina.com}

\author{Erxiao Wang}
\address{Department of Mathematics, Hong Kong University of Science and
Technology, Clear Water Bay, Kowloon, Hong Kong}
\email{maexwang@ust.hk}

\date{}
\begin{abstract}
 We will first clarify the loop group formulations for both hyperbolic
 and elliptic definite affine spheres in $\R^3$. Then we classify the rational  elements with 3 poles or 6 poles in a real twisted loop group, and compute dressing actions of them on such surfaces. Some new examples with pictures will be produced at last. 
\end{abstract}

\maketitle


\ms \hskip 3in \today


\section{Introduction}
All intuitive geometrical notions about position, size and shape developed in school geometry are invariant under rigid motions comprised of  translations, rotations and reflections. Following Klein's famous Erlangen Program (1872),  affine geometry studies the constructs invariant under the affine group $\mathrm{A}(n)=\GL(n)\ltimes \R^n$ comprised of nondegenerate linear maps and translations. While two fundamental measurements angle and distance in Euclidean geometry are no longer invariant under affine motions, notions for some relative positions such as parallels and midpoints (more generally center of mass) still make sense. As the standard directional derivative operator $D$ is preserved by $\mathrm{A}(n)$, differential geometry of curves and surfaces can also be generalized. Similarly, the symmetry group for equiaffine geometry is $\mathrm{SA}(n)=\SL(n)\ltimes \R^n$ preserving in addition the standard volume form on $\R^n$.

In 1907 Gheorghe Tzitz\'{e}ica discovered  a particular
class of hyperbolic surfaces in $\R^3$ whose Gauss curvature at any point $p$ is proportional
to the fourth power of the distance from a fixed point to the tangent plane at $p$.
He used the  structure equation
\begin{equation} \label{eqtzi}
\omega_{xy}=e^\omega-e^{-2\omega}
\end{equation}
to describe the (indefenite) affine spheres. The reader may refer to \cite{Fox12} for a concise survey on affine spheres and to \cite{Lof10} for an extensive survey on their relations to real Monge-Amp\`{e}re equations, projective structures on manifolds, and Calabi-Yau manifolds. Here we would like to emphasize the Monge-Ampr\`{e}re equation (see Calabi \cite{Eugenio Calabi.{1972}}). The graph of a locally strictly convex function $f$ is a mean curvature H affine sphere centered at the origin or infinity if and only if the Legendre transform $u$ of
$f$ solves:
\begin{equation}\label{eqMA}
\det \left(\frac{\partial^2 u}{\partial y_i \partial y_j}\right) = \left\{ \begin{array}{ll}
(H u )^{-n-2}, & \textrm{if $H \neq 0$,}\\
1, & \textrm{if $H=0$.}
\end{array} \right.
\end{equation}
Cheng \& Yau \cite{CheYau86} showed that on a bounded convex domain there is for $H < 0$ a unique negative convex solution of \eqref{eqMA} extending continuously to be $0$ on the boundary. This eventually leads to a beautiful geometric picture of complete hyperbolic ($H < 0$) affine spheres (conjectured by Calabi): such a hypersurface is always asymptotic to a convex cone with its vertex at the center and the interior of any regular convex cone is a disjoint union of complete hyperbolic affine spheres asymptotic to it with all negative mean curvature values and with centers at the vertex. Even in $\mathbb{R}^3$, it is highly nontrivial to see the above equation is equivalent to the usual surface structure equations (see Simon \& Wang \cite{Sim93}):
\begin{equation} \label{eqdas}
  \begin{cases}
     \psi_{z\bar{z}} + \frac{H}{2} e^{\psi} + |U|^2  e^{-2\psi} = 0, \\
     U_{\bar{z}}  = 0 ,
  \end{cases}
\end{equation}
where $e^\psi |\mathrm{d}z|^2$ is the affine metric and $U \mathrm{d} z^3$ is equivalent to the affine cubic form. One may also fix a holomorphic cubic form $\mathbf{U}$ (locally $\mathbf{U} = U \mathrm{d} z^3$)
and a conformal metric $g$ on a Riemann surface (such that the affine metric $e^\psi |\mathrm{d}z|^2 = e^\phi g$). Then
a coordinate-free version of \eqref{eqdas} is:
\begin{equation} \label{eqglobal}
 \Delta\phi + 4\|\mathbf{U}\|^2 e^{-2\phi} + 2H e^\phi - 2\kappa = 0,
\end{equation}
where $\Delta$ is the Laplace operator of $g$, $\| \cdot \|$ is the induced norm on cubic differential, and $\kappa$ is the Gauss curvature.

The classical Tzitz\'{e}ica equation \eqref{eqtzi} was rediscovered in many mathematical and physical contexts afterwards.(see, e.g., \cite{Dod77}, \cite{Dun02}, \cite{Gaf84}). In recent years, techniques from soliton theory have been applied to this equation by: Rogers \& Schief in the context of gas dynamics (\cite{Rog94}), Kaptsov \& Shan'ko on multi-soliton formulas (\cite{Kap97}), Dorfmeister \& Eitner on Weierstrass type representation (\cite{DoEi01}), Bobenko \& Schief on its discretizations (\cite{Bob99}, \cite{Bob991}), and Wang (\cite{Wa06}) on dressing actions following Terng \& Uhlenbeck (\cite{Ter00}) approach.

The equation \eqref{eqdas} for nonzero $U$ is an elliptic version of the classical Tzitz\'eica equation, or the Bullough-Dodd-Jiber-Shabat equation, or the affine $\mathfrak{a}_2^{(2)}$-Toda field equation, or the first $\SL(3,\mathbb{R})/\SO(2,\mathbb{R})$ elliptic system (see Terng \cite{Te08}).  Dunajski \& Plansangkate \cite{DuPl09} have given gauge invariant characterization of \eqref{eqdas} in terms of $\mathrm{SU}(2,1)$ Hitchin equations or reduction of Anti-Self-Dual-Yang-Mills equations, and linked the local equation around a  double pole of $U$ to Painlev\'e $III$. Zhou \cite{Zhou09} have given Darboux transformations for two dimensional elliptic affine Toda equations. Hildebrand \cite{Roland Hildebrand.{2013}} have given analytic formulas for  complete hyperbolic affine spheres asymptotic to semi-homogenous convex cones.

In the background of mirror symmetry, Loftin, Yau \& Zaslow \cite{LoYauZa05} constructed some global hyperbolic/elliptic affine sphere immersions from Riemann sphere minus 3 points, also called "trinoids". This has motivated our extensive studies of such surfaces from integrable system or soliton theory. In this paper we compute the  dressing actions and soliton examples. The Permutability  Theorem and group structure of dressing actions will be presented  in a subsequent  paper \cite{Wang lin Wang}. Their Weierstrass or DPW representation has been studied in \cite{Dor Wang01}, using an Iwasawa decomposition of certain twisted loop group. The equivariant solutions have been constructed in \cite{Dor Wang02}. The conformal parametrization of Hildebrand's complete affine spheres will be presented in \cite{lin Wang}.

The rest of the paper is organized as follows. In section 2, we will recall the fundamental notions and theorems of equiaffine differential geometry, and the definitions of affine spheres.  We specialize to two dimensional proper definite affine spheres in section 3 and give the loop group descriptions of them.  We then classify the simple rational elements in certain twisted  loop group in section 4 and compute the dressing actions of them on such surfaces in section 5. Some examples will be presented in the last section.

\section{Fundamental theorem of affine differential geometry  and affine spheres }
Classical affine differential geometry studies the properties of an immersed hypersurface $r: M^n \rightarrow \R^{n+1}$ invariant under the equiaffine transformations $r \to Ar+b$, where $A \in \SL_{n+1}(\R)$ and $b \in \R^{n+1}$. Here $\R^{n+1}$ is viewed as an affine space equipped with standard equiaffine structure: the canonical connection $D$ and the parallel volume form given by the standard determinant. The following form in local coordinates $(u_1,\cdots,u_n)$ is naturally an equiaffine invariant:
\begin{equation}\label{eqfun}
\Lambda = \sum_{i,j} \det \left( \frac{\partial^2 r}{\partial u_i \partial u_j}, \frac{\partial r}{\partial u_1}, \cdots, \frac{\partial r}{\partial u_n} \right) (\ud u_i \ud u_j)\otimes (\ud u_1 \wedge \cdots \wedge \ud u_n)
\end{equation}
which is independent of choice of local coordinates and thus defines a global quadratic form with values in the line bundle of top forms on $M$. Denote the determinant coefficients in \eqref{eqfun} by $\Lambda_{ij}$ and their determinant by $\det\Lambda$. Assuming $\det\Lambda \neq 0$ everywhere, a pseudo-Riemannian metric $g$ is induced uniquely up to sign (depending on the orientation of $M$) by the equation
\[
\Lambda = g \otimes \mathrm{vol}(g) \qquad \textrm{  ( i.e. volume form of $g$ )} ,
\]
or defined explicitly by
\[
g= \pm |\det\Lambda|^{ -\frac{1}{n+2} } \  \sum_{i,j} \  \Lambda_{ij} \ud u_i \ud u_j ,
\]
called the \textbf{affine metric}.  The hypersurface is said to be {\bf definite\/} or {\bf indefinite\/} if this metric $g$ is so. A smooth hypersurface is definite if and only if it is locally strictly convex. Although the Euclidean angle is not invariant under affine transformations, there exists an invariant transversal vector field $\xi$ along $r(M)$ defined by $\triangle^g r \, / \, n $, called the {\bf affine normal}. Here $\triangle^g$ is the Laplace-Beltrami operator of $g$. Another way to find the affine normal up to sign is by modifying the scale and direction of any transversal vector field (such as the Euclidean normal) to meet two natural characterizing conditions:
\begin{itemize}
\item[(i)] $D_X \xi$ or $\ud \xi (X) $ is tangent to the hypersurface for any $ X \in T_p M$,
\item[(ii)] $\xi$ and $g$ induce the same volume measure on $M$:
$$ \det \left( r_{\ast} X_1, \cdots, r_{\ast} X_n, ~\xi \right)^2= | \det ( g(X_i, X_j) ) | $$
for any $ X_i \in T_p M$.
\end{itemize}
The formula of Gauss gives the following decomposition into tangential and transverse components:
\begin{equation}\label{eqde}
    D_X r_{\ast}Y = r_{\ast}(\nabla_X Y) + g(X,Y)\xi ~,
\end{equation}
which induces a torsion-free affine connection $\nabla$ on $M$, called \textbf{Blaschke connection}. Let $\hat{C}$ denote the difference tensor between the induced Blaschke connection $\nabla$ and $g$'s Levi-Civita connection $\nabla^g$. The affine cubic form
measures the difference between the induced Blaschke connection $\nabla$ and $g$'s Levi-Civita connection
$\nabla^g$:
\begin{equation}\label{eqcu}
    C(X,Y,Z):= g(\nabla_X Y-\nabla^g_X Y,Z).
\end{equation}
It is actually symmetric in all 3 arguments and is a 3rd order invariant. The \textbf{Pick invariant} $J$ is simply $\|C\|_g^2 / (n^2-n)$.

Similar to the Euclidean case, the \textbf{affine shape operator} $S$ defined by the formula of Weingarten:
\[
 D_X \xi = -r_{\ast}(S (X)) ,
 \]
is again self-adjoint with respect to $g$.
The {\bf affine mean curvature\/} $H$ and the {\bf affine Gauss curvature\/} $K$ are defined as $ H = \Tr S \, / \, n , ~  K = \det S $.

Note that the affine normal is also a 3rd order invariant. Hence $S$ is a 4th order invariant, which can actually be computed from $g$ and $C$:
\begin{eqnarray}\label{eqgcs}
  H & =&  R_g - J , \qquad \textrm{ where } R_g = \textrm{ scalar curvature of } g. \\
  \nonumber  g(S_0(X),Y) & =&  -\frac{2}{n} \Tr \{ Z \mapsto (\nabla^g_Z \hat{C})(X,Y) \} , \qquad \textrm{ where } S_0 := S - H \cdot \mathrm{Id}.
\end{eqnarray}

The readers may refer to several textbooks \cite{Bla23, LSZ93, NomSas94} for more details of the above basic notions.

\begin{thm}[Dillen-Nomizu-Vrancken type fundamental theorem \cite{DNV90}]\label{fun}
 Given a nondegenerate symmetric $2$-form $g$ and a symmetric $3$-form $C$ on simply connected $M$ satisfying two compatibility conditions:  the apolarity condition $\Tr_g C = 0 $, and the conjugate connection $\overline{\nabla} :=\nabla^g - \hat{C}$ is projectively flat,  there exists a global immersion into $\R^{n+1}$ unique up to equiaffine motions such that $(g,C)$ are the induced affine metric and cubic form respectively.
\end{thm}
\rem
When we scale the immersion $f$ to $ \rho f$ for some positive constant $\rho$, the above invariants changes by:
\begin{eqnarray}\label{scale}
\nonumber \nabla &\rightarrow& \nabla, \qquad g \rightarrow \rho^{\frac{2(n+1)}{n+2}} g,  \qquad \xi \rightarrow \rho^{\frac{-n}{n+2}} \xi,  \qquad C \rightarrow \rho^{\frac{2(n+1)}{n+2}} C, \\
S &\rightarrow& \rho^{\frac{-2(n+1)}{n+2}} S \qquad H \rightarrow \rho^{\frac{-2(n+1)}{n+2}} H, \qquad K \rightarrow \rho^{\frac{-2n(n+1)}{n+2}} K .
\end{eqnarray}

Historically the first important class of affine surfaces is the affine spheres
studied by Tzitz\'{e}ica in a series of papers from 1908 to 1910. A {\bf proper
affine sphere} is a surface whose affine normal lines all meet in a point (the
center), such as all ellipsoids and hyperboloids. An {\bf improper affine sphere}
is a surface whose affine normals are all parallel, such as all paraboloids and
all ruled surfaces of the form $x_{3}=x_{1}x_{2}+f(x_{1})$. An equivalent definition is
$S=H\cdot$ Id, which implies from the compatibility conditions that the affine
mean curvature $H$ must be constant. When $S\equiv 0, ~ \xi$ is a constant vector and
it is improper. So improper affine spheres are special affine maximal surfaces $(H=0)$ and have been integrated explicitly  for dimension 2 in \cite{Bla23}. When $S=H\cdot$ Id $\neq 0$, it
is proper and $\xi=-H(x-x_{0})$ with some $x_{0}$ being the center. For simplicity,
we will always make $x_{0}=0$ by translating the surface. In the indefinite case,
we can scale the surface and change the sign of $\xi$ if necessary to normalize $H=-1$. However, in the definite case we have a preferred choice of the signs of $g$ and $\xi$ by requiring $g$ to be positive definite, or equivalently by choosing $\xi$ to point to the inside of the convex surface. So we need to distinguish two
cases according to the sign of $H$:\par
(1) the {\bf elliptic affine spheres} whose center is inside the convex surface $(H>0)$, such as all ellipsoids;\par
(2) the {\bf hyperbolic affine spheres} whose center is outside the convex surface $(H<0)$, such as all hyperboloids of two sheets and $x_1x_2x_3 = 1$. \par
Improper definite affine spheres $(H=0)$, are naturally called {\bf parabolic
affine spheres}, such as all elliptic paraboloids. While quadrics are the only global examples of elliptic and parabolic affine spheres (by Calabi, Cheng-Yau, J\"{o}rgens, Pogorelov), there are many global hyperbolic affine spheres asymptotic to each sharp (or regular) convex cone (by Cheng-Yau). One of the main goals of this paper is to construct families of local examples of both elliptic and hyperbolic affine spheres in $\R^3$ using integrable system techniques.

\section{Loop group description of definite affine spheres}
\ms
From now on we consider only the surface case in $\R^3$ where its affine metric $g$ is positive definite: this means that  $r(M)$ is locally strongly convex and oriented so that $\Lambda$ is positive valued. In particular its Euclidean Gauss curvature is positive and the affine normal $\xi$ is chosen to point to the concave side of the surface. This essentially induces a unique Riemann surface (or complex) structure on $M$ in whose coordinates $g=e^\psi|\ud z|^2$ and the orientation given by $\mathrm{i} \ud z \wedge \ud \bar{z}$ is consistent with the orientation induced by $\xi$. Note that $-r$ would have the same fundamental invariants as $r$ but reverse the orientation. So it is not considered to be equi-affinely equivalent to $r$. Indeed $3\times 3$ negative identity matrix has determinant $-1$ and is not in $\SL (3,\C)$. Alternatively we are studying ``\textbf{affine-conformal}'' immersions of any Riemann surface $M$ into $\R^3$:
\begin{equation}\label{eqconf}
\mid r_{z} \quad r_{\bar{z}} \quad r_{zz} \mid \  = \  0  \quad \qquad \textrm{ affine-conformal condition. }
\end{equation}
Then the following two determinants completely determine the fundamental invariants:
\[
\mid r_{z} \quad r_{\bar{z}} \quad r_{z \bar{z} } \mid \  = \  \frac{\mathrm{i}}{4} e^{2\psi}, \qquad
\mid r_{z} \quad r_{zz} \quad r_{zzz } \mid \  = \  \frac{\mathrm{i}}{4} U^2,
\]
with $g=e^\psi |\ud z|^2 $ and $C=U\ud z^3+\bar{U}\ud \bar{z}^3 $. This can also be illustrated by computing the evolution equations for the positively oriented frame $ \tilde{F} = ( r_z, r_{\bar{z}}, \xi = 2 e^{-\psi} r_{z \bar{z}} ) $:
\begin{equation}\label{eqf1}
\tilde{F}^{-1} \ud \tilde{F} = \left(
                                 \begin{array}{ccc}
                                   \psi_z \ud z & \bar{U} e^{-\psi} \ud \bar{z}  & -H \ud z + 2e^{-2\psi} \bar{U}_z \ud \bar{z} \\
                                   U e^{-\psi} \ud z & \psi_{\bar{z}} \ud \bar{z} & -H \ud \bar{z} + 2e^{-2\psi} U_{\bar{z}} \ud z \\
                                   \frac{1}{2}e^\psi \ud \bar{z} & \frac{1}{2}e^\psi \ud z & 0 \\
                                 \end{array}
                               \right).
\end{equation}
The compatibility condition ($\tilde{F}_{z\bar{z}} = \tilde{F}_{\bar{z}z}$) or the flatness of $\tilde{F}^{-1} \ud \tilde{F}$ determines every entry in the above matrix one-form satisfying two structure equations (first derived by Radon):
\begin{eqnarray}
  H &=& -2 e^{-\psi} \psi_{z\bar{z}} - 2 |U|^2 e^{-3\psi}  \qquad\qquad \textrm{ Gauss equation} \label{com1}\\
  H_{\bar{z}} &=& 2 e^{-3\psi} \bar{U} U_{\bar{z}} - 2 e^{-\psi} (e^{-\psi}\bar{U}_z)_z .  \quad\quad \textrm{ Codazzi equation for } S\label{com2}
\end{eqnarray}
Conversely, given a positive definite metric $g$ and fully symmetric cubic form $C$ as above satisfying \eqref{com1} and \eqref{com2}, there exits a definite affine surface unique up to equiaffine motions such that $(g, C)$ are the induced affine metric and cubic form respectively. This is  \emph{the fundamental theorem for definite affine surface}. Here $H$ is indeed the affine mean curvature function since the affine shape operator $S$ takes the following matrix form in terms of the basis $ ( \partial_z, \partial_{ \bar{z} } )$:
\[
S = \left(
      \begin{array}{cc}
        H & -2e^{-2\psi} \bar{U}_z  \\
        -2e^{-2\psi} U_{\bar{z}} & H  \\
      \end{array}
    \right) .
\]

In summary we obtain the governing equations for definite affine spheres in $\R^3$  (see also Simon-Wang \cite{Sim93}):
\begin{equation} \label{eqgc}
  \left\{ \begin{array}{rl}
     \psi_{z\bar{z}} + \frac{H}{2} e^{\psi} + |U|^2  e^{-2\psi} &= 0, \\
     U_{\bar{z}}  &= 0 ,
  \end{array} \right.
\end{equation}
whose solution determines a unique definite affine sphere (up to equi-affine motions) with constant affine mean curvature $H$ and affine metric $e^\psi|\ud z|^2$, by integrating the frame equation:
\begin{equation} \label{eqfr}
\tilde{\alpha}:=\tilde{F}^{-1} \ud \tilde{F} = \left(
                                 \begin{array}{ccc}
                                   \psi_z \ud z & \bar{U} e^{-\psi} \ud \bar{z}  & -H \ud z  \\
                                   U e^{-\psi} \ud z & \psi_{\bar{z}} \ud \bar{z} & -H \ud \bar{z}  \\
                                   \frac{1}{2}e^\psi \ud \bar{z} & \frac{1}{2}e^\psi \ud z & 0 \\
                                 \end{array}
                               \right).
\end{equation}

It is now clear that $U \ud z^3$ is a globally defined holomorphic cubic differential (i.e. in $H^0(M,K^3)$ where $K$ is the canonical bundle of $M$). Recall Pick's Theorem: $ C\equiv 0 $ if and only if $r(M)$ is part of a quadric surface. So $U$ is nonzero except for the quadrics. Away from its isolated zeroes one could make a holomorphic coordinate change to normalize $U$ to a nonzero constant but we will not do that now. These zeroes of $U$ will be called ``planar'' points of the affine surface.

We would like to emphasize that the immersion is analytic for any definite affine sphere, since the defining equation is elliptic (cf. \cite{Bla23} \S 76 ).

 The following observation is crucial for the integrability of definite affine spheres: the system \eqref{eqgc} is invariant under $U \rightarrow e^{i\theta} U$ with any constant $\theta \in \R/2\pi\Z $. Thus an $S^1$-family of definite affine spheres $r^\theta$ can be associated to any given one, with the same affine metric but different affine cubic forms
$C=e^{i\theta}U \ud z^3 + e^{-i\theta}\bar{U}\ud \bar{z}^3$.

To further reveal the hidden symmetry, we will scale the surface to normalize $H=\pm 2$ using \eqref{scale} for elliptic or hyperbolic case, also replace $U$ by $e^{i\theta}U$ in the frame equation \eqref{eqfr}, and then use $\diag (2 \sqrt{\mp 1} \lambda^{-1} e^{-\psi/2},2 \sqrt{\mp 1} \lambda e^{-\psi/2},1)$ to gauge the Maurer-Cartan form $\tilde{\alpha}$ to:
\begin{equation} \label{eqla}
\alpha_{\lambda} =
\begin{pmatrix}
\frac{1}{2}(\psi_z \ud z -\psi_{\bar{z}} \ud \bar{z}) & \lambda^{-1}\bar{U}e^{-\psi} \ud \bar{z} &\sqrt{\mp 1} \lambda\ e^{\psi/2} \ud z \cr
\lambda U e^{-\psi} \ud z  & \frac{1}{2}(\psi_{\bar{z}} \ud \bar{z} - \psi_z \ud z ) &\sqrt{\mp 1} \lambda^{-1} e^{\psi/2} \ud \bar{z} \cr
\sqrt{\mp 1} \lambda^{-1} e^{\psi/2} \ud \bar{z} & \sqrt{\mp 1} \lambda e^{\psi/2} \ud z & 0
\end{pmatrix}
\end{equation}
Although $\a_\l$ has real geometric meaning only for $|\l|=1$, it is actually flat for all $\l \in \mathbb{C}\setminus\{0\}$. Furthermore it takes value in certain twisted loop algebra of $\sli (3,\C)$, i.e. satisfying two reality conditions:
\begin{equation}\label{eqre}
 \tau(\alpha_{1/ \bar \lambda})=\alpha_\lambda, \quad \  \sigma(\alpha_{e^{-2\pi{\rm i}/ 6} \lambda})= \alpha_{ \lambda},
\end{equation}
where $\tau(X)= T \bar{X} T^{-1}$ and $\sigma(X)= -P  X^t P^{-1}$ for
\beq\label{3i}
T=\bpm 0 & 1 & 0 \\ 1 & 0 & 0 \\ 0 & 0& -H/2 \epm, \quad P=\bpm 0 & \e^4 & 0 \\ \e^2 & 0 & 0 \\ 0 & 0& 1 \epm, \quad \e=e^{\frac{\pi i}{3}}.
\eeq
We denote $\tau_1$ for $\tau$-reality condition in hyperbolic case, $\tau_2$ in elliptic case, and $T_1$ for T in  hyperbolic case, $T_{-1}$ in elliptic case, i.e.
\[
T_1=\bpm 0 & 1 & 0 \\ 1 & 0 & 0 \\ 0 & 0& 1\epm, \; T_{-1}=\bpm 0 & 1 & 0 \\ 1 & 0 & 0 \\ 0 & 0& -1\epm.
\]

Note that $\tau$ is a conjugate linear involution of $\sli (3,\C)$ whose fixed point set is isomorphic to $\sli(3,\R)$; and $\sigma$ is an order $6$ automorphism of $\sli (3,\C)$ giving the following eigenspace decomposition or $\Z_6$-gradation:

 $$ sl(3,\C)=\oplus_{j=0}^5 \cg_j, \quad [\cg_j,\cg_k] \subset \cg_{j+k}. $$
with
\begin{align*}
\cg_0 &=\left\{ \bpm s & 0 & 0 \\ 0 & -s & 0 \\ 0 & 0& 0 \epm \right\}, \quad
\cg_1 =\left\{ \bpm 0 & 0 & s_1 \\ s_2 & 0 & 0 \\ 0 & s_1 & 0 \epm \right\}, \\
\cg_2 &=\left\{ \bpm 0 & 0 & 0 \\ 0 & 0 & s \\ -s & 0& 0 \epm \right\}, \quad
\cg_3 =\left\{ \bpm s & 0 & 0 \\ 0 & s & 0 \\ 0 & 0& -2s \epm \right\}, \\
\cg_4 &=\left\{ \bpm 0 & 0 & s \\ 0 & 0 & 0 \\ 0 & -s& 0 \epm \right\}, \quad
\cg_5 =\left\{ \bpm 0 & s_1 & 0 \\ 0 & 0 & s_2 \\ s_2 & 0& 0 \epm \right\}.
\end{align*}

We verify that $\sigma\tau=\tau\sigma$ and they define a $6$-symmetric space ``$\SL(3, \mathbb{R})"/ \SO(2,\mathbb{R})$.

The induced automorphisms on $\SL (3,\C)$ (still denoted by $\tau$ and $\sigma$) are:
\begin{equation}\label{eqreg}
\tau(g)=P\bar{g}P^{-1} , \qquad \sigma(g)= Q \ (g^t)^{-1} Q^{-1}.
\end{equation}

If we solve $F_{\lambda}^{-1} \ud F_{\lambda}= \alpha_\lambda$ uniquely with certain initial condition $F_{\lambda}(p_0)$ at any base point $p_0$ of $M$, it is easy to show that $F_{\lambda}(p_0)^{-1}F_{\lambda}$ satisfies the reality conditions \eqref{eqre} and therefore lies in the corresponding twisted loop group.

\begin{rem}\label{rem-reframe}
Hereafter we will always choose the initial loop $F_{\lambda}(p_0)= \mathrm{I} $. Then we may conjugate the complex frame to a real $\SL (3,\C)$-frame:
\[
F^{\R} := \mathrm{Ad} \begin{pmatrix} \frac{1}{\sqrt{2}} & \frac{1}{\sqrt{2}} & 0 \\ \frac{\mathrm{i}}{\sqrt{2}} & \frac{-\mathrm{i}}{\sqrt{2}}& 0 \\ 0 & 0 & \sqrt{H/2} \end{pmatrix} \cdot  F_\lambda.
\]
In fact $F^{\R}=(e_{1},e_{2},\xi)$ with $\{e_{1},e_{2}\}$ being simply an orthonormal tangent frame w.r.t. the affine metric. So we obtain an affine sphere immersion $r=-H^{-1} \xi = \mp \xi$. It is clear now that we may also simply take the real part of the last column of $F_\lambda$ to get an equivalent affine sphere modulo affine motions.
\end{rem}

 As in \cite{DoEi01} and \cite{Wa06}, we write

\begin{equation*}
 P= \bpm 0 & \e^4 & \\ \e^2 & 0 & \\ & & 1\\ \epm
 =\bpm \e^4 & & \\ & \e^2 & \\ & & 1\\\epm
 \bpm 0 & 1& \\ 1 & 0 & \\ & & 1 \\ \epm=Q P_{12}=P_{12} Q^{-1},
\end{equation*}

 where
 \beq\label{3n}
 Q=\bpm \e^4 & & \\ & \e^2 & \\ & & 1\\\epm , \quad
 P_{12}=\bpm 0 & 1& \\ 1 & 0 & \\ & & 1 \\ \epm.
 \eeq
 It follows that
\begin{eqnarray} \label{3sigma12}
 \sigma(g) &=&P(g^t)^{-1}P =Q P_{12}(g^t)^{-1}P_{12} Q^{-1}=\sigma_2 (\sigma_1(g)), \notag \\
 \sigma_1 (g(\l)) &=& g(-\l), ~~ \sigma_2 (g(\l)) = g(\epsilon^{-2} \l).
 \end{eqnarray}
where
 \beq \label{3o}
 \sigma_1(g)=P_{12} (g^t)^{-1}P_{12}, \quad \sigma_2(g)=QgQ^{-1},
 \eeq
 and $\sigma=\sigma_1\circ \sigma_2 =\sigma_2\circ \sigma_1$.

\ms

 Let $\C^{\times}:=\C\setminus\{0\}$. We adopt the following notations for loop groups: \par
 $\Lambda G=\{$ holomorphic maps from $\C^{\times}\cap (I_{r}\cup I_{\frac{1}{r}})$ to $G \}$,\par
 $\Lambda_{\mathbb{C}^{\times}} G=\{$ holomorphic maps from $C^{\times}$ to $G \}$,\par
 $\Lambda_{I} G=\{$ holomorphic maps from I to $G$ \},\par
 $\Lambda^{\sigma,\tau} G=\{$  $g\in\Lambda G$: $\tau(g(\bar \l^{-1})) = g(\l)$, $\sigma(g(\l)) = g(\e \l)\}$,\par
 $\Lambda^{\sigma,\tau}_{\mathbb{C}^{\times}} G=\{$  $g\in\Lambda_{\mathbb{C}^{\times}} G$: $\tau(g(\bar \l^{-1})) = g(\l)$, $\sigma(g(\l)) = g(\e \l)\}$,\par
 $\Lambda^{\sigma,\tau}_{I} G=\{$  $g\in\Lambda_{I} G$: $\tau(g(\bar \l^{-1})) = g(\l)$, $\sigma(g(\l)) = g(\e \l)\}$,\par
 where $0<r<1$ is sufficiently small and
 \begin{align*}
 I_{r}=\{\l\in\C: |\l|<r\},  \quad I_{\frac{1}{r}}=\{r\in \C\cup\{\infty\}: |\l|>\frac{1}{r}\}, \quad I = I_{r}\cup I_{\frac{1}{r}}.
 \end{align*}

 Similar notations can apply to their Lie algebras. Then $\alpha_\lambda$ in \eqref{eqfr} is a
$\Lambda_{\mathbb{C}^\times}^{\sigma,\tau} \sli (3,\C) $-valued flat connection, and the corresponding frame
$F_{\lambda}$ lies in $\Lambda_{\mathbb{C}^\times}^{\sigma,\tau} \SL (3,\C)$ for any $0<r<1$. It is remarkable that a simple ``algebraic'' condition characterize such extended frames of proper definite affine spheres:

\begin{thm}[Loop group formulation for proper definite affine spheres]\label{loop}
Let $F(z,\bar{z})$ be any smooth map from a domain in $\C$ to the twisted loop group
$\Lambda_{\C^\times}^{\sigma,\tau} \SL_3 (\C) $. If $F^{-1}F_{z}$ is linear in the loop parameter $\lambda$, i.e. of the form $A \lambda + B$, and $A_{13}$ is nowhere zero, then $F^{-1}\ud F$ is gauge equivalent to the Maurer-Cartan form \eqref{eqla} of proper definite affine spheres. In other words, $F$ is gauge equivalent to the extended frame of a proper definite affine sphere if and only if $[F]$ defines a primitive harmonic map into the $6$-symmetric space $\textrm{``}\SL_3 (\R) \textrm{''}/ \SO_2 (\R)$ with some nondegeneracy condition.
\end{thm}
\begin{proof}
For simplicity, we only show the hyperbolic ($H=-1$) case in the flat connection \eqref{eqla}. The positive case is completely parallel.

The reality conditions \eqref{eqre} guarantee that $F^{-1}F_z$ must be linear in $\lambda$. So we have
\begin{equation} \label{eqloop}
F^{-1}F_{z} = A \lambda + B, \qquad F^{-1}F_{\bar{z}} = C^{-1} \lambda + D,
\end{equation}
with $A \in \fg_{-1}$, $B \in \fg_0$, $C=\tau(A)$, and $D=\tau(B)$. The fixed points of both $\sigma$ and $\tau$ are of the form $\diag (e^{i\beta}, e^{-i\beta}, 1 )$. Gauging by them respect the reality conditions. Let $e^{i\beta}=\pm \frac{A_{13}}{|A_{13}|}$. Use it to gauge $A_{13}$ to a real positive function which is then set to $e^{\psi/2}$. The rest follows from the equations of flatness.
\end{proof}
\section{Simple elements in  $\Lambda^{\sigma,\tau}_{I} GL(3,\C)$}

A rational element in any loop group is usually called \textbf{a simple element} if it has the least number of simple poles. To construct the simple elements in the twisted loop group $\Lambda^{\sigma,\tau}_{I} GL(3,\C)$, we will handle order 3 $\sigma_2$-twisting first:
 \blem ([7])
      A simple element  $g\in\Lambda^{\sigma_{2}}_{I} GL(3,\C)$ has three simple poles and always take  the following form:
  \beq \label{4a}
  g(\l)=A \,(I+\frac{R}{\l -\a}+\frac{\e^2Q^{-1}RQ}{\l -\e^2\a} +\frac{\e^4 QRQ^{-1}}{\l -\e^4\a}),
  \eeq
  where A is diagonal and
  \[
  Q=diag(\epsilon^4,\epsilon^2,1), \quad \e=e^{\frac{\pi i}{3}}, \quad r < |\a| < 1/r.
  \]
 \elem
 \ms

To construct simple element in $\Lambda^{\sigma}_{I} GL(3,\C)$, plug  $\eqref{4a}$  into $\sigma_1$-reality condition in \eqref{3sigma12}, we obtain that $g \in \Lambda^{\sigma}_{I} GL(3,\C)$ iff
 $$ g(-\l) \; P_{12} \; g^t(\l) =P_{12}, $$ where $P_{12}$ is given in $\eqref{3n}$, i.e.
\beq \label{4b}
g(-\l)\; P_{12} \; \left( I+\frac{R^t}{\l -\a}+\frac{\e^2QR^t Q^{-1}}{\l -\e^2\a} +\frac{\e^4 Q^{-1}R^t Q}{\l -\e^4\a} \right)A^t=P_{12}.
 \eeq
Compute the LHS of $\eqref{4b}$ at $\l =\infty$ to get
 $$A\;P_{12}\;A^t =P_{12}.$$
Write $A=\diag (d_1,d_2,d_3)$, and we have
$$ d_1 d_2 =1, \quad d_3^2 =1.$$  So
\begin{equation} \label{4c}
A=\diag(d, d^{-1},\pm1),
\end{equation}
where $d\in \C^{\times}$.

Compute the residue of the LHS of \eqref{4b} at $\l=\a$ to get
\beq \label{4d}
g(-\a) \;P_{12}\; R^t A^t=0, \eeq
i.e.
$$ A(I+\frac1\a (-\frac{R}{2}+\e^{-2}Q^{-1}RQ+\e^2 QRQ^{-1}))\;P_{12}\;R^t A^t=0. $$
Write $R=(b_{ij})_{1\leq i,j\leq3}$. Then direct computation implies
\beq \label{4e}
\bpm d & & \\ & d^{-1} & \\ & & \pm1\\ \epm \left\{I+\frac1\a \left( -\frac32 R +3
\bpm 0 & 0 & b_{13} \\ b_{21} & 0 & 0 \\ 0 & b_{32} & 0 \\ \epm \right)   \right\}\; P_{12}\;R^t A^t=0.
 \eeq
 From $\eqref{4e}$ we deduce that $rank(R)=3$ is impossible. If $rank(R)=0$, we get $g(\l) = A$ is trivial.  If $rank(R)=1$, we assume that
 \beq \label{4f}
 R=\bpm b_2 \\c_2 \\ e_2 \epm \bpm b_1, & c_1, & e_1 \epm.   \eeq
 Substituting $\eqref{4f}$ into $\eqref{4e}$ and computing directly, we have
 \begin{equation} \label{4f2}
 R= \frac{2\a}{3}\bpm \frac{c_1}{2b_1c_1 -1} \\ b_1 \\ 1 \epm \bpm b_1, & c_1, & 1 \epm, \quad  2b_1c_1 -1 \not=0.
 \end{equation}
 Residue of LHS of $\eqref{4b}$ at $\l=\e^2 \a$ gives
 $$g(-\e^2\a) \;P_{12}\;QR^t Q^{-1} A^t=0,$$
 which is equivalent to
 $$ Q g(-\e^2\a) Q^{-1} \;P_{12}\;R^t A^t=0,$$
 i.e. $$\sigma_2(g(-\e^2\a)) \;P_{12}\;R^t A^t=0.$$
 This is equivalent to $\eqref{4d}$. Residue at $\l=\e^4 \a$ is also equivalent to $\eqref{4d}$. So the loop group element of \textbf{rank 1 type} is as follows:
 \beq \label{4g}
g(\l)=  \bpm d & & \\ & d^{-1} & \\ & & \pm1\\ \epm \left[ I + \frac{2}{\l^3-\alpha^3} \bpm \frac{\a^{3}b_1c_1}{2b_1c_1-1} & \frac{\a\l^{2}{c_1}^{2}}{2b_1c_1-1}& \frac{\a^{2}\l c_1}{2b_1c_1-1} \\ \a^2\l b_{1}^{2}& \a^{3}b_1c_1 & \a\l^2 b_1\\ \a\l^2 b_1 & \a^2\l c_1 & \a^3\epm   \right] ;
\eeq

 By explicit computation, if $rank(R)=2$, we have
\[
 R=\frac{2\a}{3}
  \bpm  \frac{b_1c_1-1}{2b_1c_1-1} & \frac{-{c_1}^{2}}{2b_1c_1-1}& \frac{c_1}{2b_1c_1-1} \\  b_{1}^{2}& (1- b_1c_1) & -b_1\\ -b_1 &  c_1 & 0\epm, \quad  2b_1c_1 -1 \not=0,
\]
with the corresponding \textbf{rank 2 type} loop group element
\beq \label{4h}
g(\l) =  \bpm d & & \\ & d^{-1} & \\ & & \pm1\\ \epm \left[ I + \frac{2}{\l^3-\alpha^3} \bpm \frac{\a^{3}(b_1c_1-1)}{2b_1c_1-1} & \frac{-\a\l^{2}{c_1}^{2}}{2b_1c_1-1}& \frac{\a^{2}\l c_1}{2b_1c_1-1} \\ \a^2\l b_{1}^{2}& \a^{3}(1-b_1c_1) & -\a\l^2 b_1\\ -\a\l^2 b_1 & \a^2\l c_1 & 0\epm   \right].
\eeq
We also compute that the determinant of $g(\l)$ for the both cases are $[(\l^3+
\alpha^3)/ (\l^3-\alpha^3)]^{rank(R)}$, i.e. only depending on the poles and the rank of the residues.

Let $l :=(b,c,1)$, $\ell$ be the line $\mathbb{C} \cdot l$, and introduce the following `cone':
\[
\Delta:=\{(z_1,z_2,z_3) \in \mathbb{C}^3 \; | \; 2z_1z_2=z_3^2, \textrm{or} \; z_3=0\}.
\]
Then $2bc\neq1$ is equivalent to $\ell\nsubseteq\Delta$. We observe that $\text{Image} (\text{Res}_{\a}\; g^t) $ for rank 1 type is and $ \text{Kernel}(\text{Res}_{\a}\; g \; P_{12})$ for rank 2 type are both $\ell^t:=\mathbb{C}\cdot l^t $ .
\begin{notation}
Since a line $\ell$ not in $\Delta$, a complex number $\a$ and a diagonal matrix A  \eqref{4c} determine $g(\l)$ uniquely in both types, we always use $A g_{\a,\ell}(\l)$ to denote the rank 1 type element \eqref{4g} and use $A m_{\a,\ell}(\l)$ to denote the rank 2 type element \eqref{4h}.
\end{notation}
We summarize the above computations and some basic but useful facts into the following proposition:

 \bprop \label{p42}
Any simple element in $ \Lambda^{\sigma}_{I} GL(3,\C)$ is either $A g_{\a,\ell}(\l)$ of rank 1 type \eqref{4g} or $A m_{\a,\ell}(\l)$ of rank 2 type \eqref{4h},  and they have the following properties:   \\
(1)~~the determinant  is $[(\l^3+\alpha^3)/ (\l^3-\alpha^3)]^{rank(R)}$, independent of $\ell$; \\
(2)~~$A m_{\a,\ell}(\l) = \frac{\l^3+\a^3}{\l^3-\a^3} A g_{\a,\ell}(-\l)$; \\
(3)~~$  A \ g_{\a,\ell}(\l) = g_{\a,\ell A^{-1}}(\l) \ A.$
\eprop

\ms
Finally, we consider the $\tau$-reality condition. Recall  $g(\l) \in \Lambda^{\sigma}_{I} GL(3,\C)$  satisfies $\tau$-reality condition iff
 $$ g(\l)=\tau(g(\bar \l^{-1})) = T \overline{g(\bar \l^{-1})} T^{-1}.$$
If $g$ has a pole in $\a$ , it must also have a pole in $\bar \a^{-1}$. So $|\a|$ must be $1$ if the above simple element $A g_{\a,\ell}(\l)$ or $A m_{\a,\ell}(\l)$ satisfies $\tau$-reality condition. We will treat hyperbolic case ($\tau=\tau_1$) and elliptic case ($\tau=\tau_2$) separately.

(1) For hyperbolic case, we first consider the rank 2 type $A m_{\a,\ell}(\l) $, then
\beq \label{4i}
 \bar A \bar m_{\a,\ell}(\bar\l^{-1})  m_{\a,\ell}^{t} (-\l) A^t = I .
 \eeq
 We divide our computations into the following steps: \\
(i) Evaluate LHS of $\eqref{4i}$ at $\l=\infty$ to get
 $$ \bar A (I-3\bar \a^{-1} \diag(\bar R)) A^{t} =I, $$
 where $\diag(M)$ is the diagonal part of a matrix $M$. This implies that
 \begin{align*}
 \bpm \frac{\bar d}{2\bar b_1 \bar c_1 -1} & &  \\ & \bar d^{-1} (2\bar b_1 \bar c_1-1) & \\ & & \pm 1 \epm
  \bpm d & & \\ & d^{-1} & \\ & & \pm 1\epm
  =I.
 \end{align*}
 It follows that
 \begin{align} \label{4j}
 |d|^2 =(2\bar b_1\bar c_1 -1).
 \end{align}
(ii) Compute the residues of LHS of $\eqref{4i}$ at $\l=-\a$ and $\l=\bar \a^{-1}$ to get
 $$  \overline{A m_{\a,\ell}(-\bar\a^{-1})}R^{t}  =0,$$
i.e.
 \beq \label{4k}
  m_{\a,\ell}(-\bar\a^{-1}) \bar R^{t} =0.
  \eeq
 Residues of LHS of $\eqref{4i}$ at $\l=-\e^2\a,\ -\e^4 \a$ give the same condition $\eqref{4k}$. From \eqref{4k}, we have:
 \[
 \left[ I + \frac{2}{-\alpha^3-\alpha^3} \bpm \frac{\a^{3}(b_1c_1-1)}{2b_1c_1-1} & \frac{-\a\l^{2}{c_1}^{2}}{2b_1c_1-1}& \frac{\a^{2}\l c_1}{2b_1c_1-1} \\ \a^2\l b_{1}^{2}& \a^{3}(1-b_1c_1) & -\a\l^2 b_1\\ -\a\l^2 b_1 & \a^2\l c_1 & 0\epm   \right]
 \;\; \overline{\bpm  \frac{b_1c_1-1}{2b_1c_1-1} & \frac{-{c_1}^{2}}{2b_1c_1-1}& \frac{c_1}{2b_1c_1-1} \\  b_{1}^{2}& (1- b_1c_1) & -b_1\\ -b_1 &  c_1 & 0\epm^t}=0
 \]
i.e.
\[
\bpm  \frac{b_1c_1}{2b_1c_1-1} & \frac{{c_1}^{2}}{2b_1c_1-1}& \frac{c_1}{2b_1c_1-1} \\  b_{1}^{2}&  b_1c_1 & b_1\\ b_1 &  c_1 & 1\epm
\bpm  \ms \overline{\frac{b_1c_1-1}{2b_1c_1-1}} &\overline{ b_{1}^{2}}& \overline{-b_1 }\\  \ms \overline{\frac{-{c_1}^{2}}{2b_1c_1-1}} & \overline{(1- b_1c_1)} &\overline{c_1} \\\overline{\frac{c_1}{2b_1c_1-1}} &\overline{-b_1} & 0\epm=0
\]
i.e.
\[
(b_1,c_1,1) \;\; \bpm  \overline{b_1c_1-1} &\overline{ b_{1}^{2}}& \overline{-b_1 }\\ \overline{-{c_1}^{2}} & \overline{(1- b_1c_1)} &\overline{c_1} \\\overline{c_1} &\overline{-b_1} & 0\epm=0.
\]
Direct computation implies that:
\begin{align} \label{4l}
c_1=\bar b_1.
\end{align}
Combine \eqref{4j} and \eqref{4l}, we get that $A m_{\a,\ell}(\l) \in \Lambda^{\sigma,\tau_1}_{I} GL(3,\C)$ iff
\[
|\a|=1,\;\;\ c_1=\bar b_1, \;\;\;  |d|^2 =(2\bar b_1\bar c_1 -1)=(2 | b_1|^2 -1)>0.
\]
Then $A m_{\a,\ell}(\l)$ can be written as
\begin{align} \label{4l2}
A m_{\a,\ell}(\l) =\bpm d & & \\ & d^{-1} & \\ & & \pm1\\ \epm \left[ I + \frac{2}{\l^3-\alpha^3} \bpm \frac{\a^{3}(|b|^2-1)}{2|b|^2-1} & \frac{-\a\l^{2}{\bar b}^{2}}{2|b|^2-1}& \frac{\a^{2}\l \bar b}{2|b|^2-1} \\ \a^2\l b^{2}& \a^{3}(1-|b|^2) & -\a\l^2 b\\ -\a\l^2 b & \a^2\l \bar b & 0\epm   \right].
\end{align}
If we consider the rank 1 type $A g_{\a,\ell}(\l)$, by the similar computation, we will get
\[
|d|^2 =(2\bar b_1\bar c_1 -1), ~~~ c_1 = - \bar b_1,
\]
i.e. $|d|^2 =-(2 |b_1|^2 +1)$, which is a contradiction. So there is no rank 1 type simple element in  $ \Lambda^{\sigma}_{I} GL(3,\C)$ satisfies the $\tau_1$-reality condition.

\ms
\ms
(2) For elliptic case, we consider the rank 1 type $A g_{\a,\ell}(\l) $, then
\begin{align} \label{4m}
\bar A \bar g_{\a,\ell}(\bar\l^{-1}) I_{2,1} g_{\a,\ell}^{t} (-\l) A^t = I_{2,1} .
\end{align}
here
\[
I_{2,1}=\bpm  1 & 0&0 \\ 0& 1 & 0\\0 & 0 & -1\epm
\]
(i) Evaluate LHS of $\eqref{4m}$ at $\l=\infty$ to get
\begin{align} \label{4n}
 |d|^2 =(2\bar b_1\bar c_1 -1)
\end{align}
(ii) Evaluate LHS of $\eqref{4m}$ at $\l=-\a$ and $\l=\bar \a^{-1}$ to get
\begin{align} \label{4o}
c_1=\bar b_1
\end{align}
Combine \eqref{4n} and \eqref{4o}, we get that $A g_{\a,\ell}(\l) \in \Lambda^{\sigma,\tau_2}_{I} GL(3,\C)$ iff
\begin{align} \label{4p}
|\a|=1,\;\;\ c_1=\bar b_1, \;\;\;  |d|^2 =(2\bar b_1\bar c_1 -1)=(2 | b_1|^2 -1)>0.
\end{align}
Then $A g_{\a,\ell}(\l)$ can be written as
\begin{align} \label{4q}
A g_{\a,\ell}(\l)
 &=\bpm d & & \\ & d^{-1} & \\ & & \pm1\\ \epm \left[ I + \frac{2}{\l^3-\alpha^3} \bpm \frac{\a^{3}|b|^2}{2|b|^2-1} & \frac{\a\l^{2}{\bar b}^{2}}{2|b|^2-1}& \frac{\a^{2}\l \bar b}{2|b|^2-1} \\ \a^2\l b^{2}& \a^{3}|b|^2 & -\a\l^2 b\\ \a\l^2 b & \a^2\l \bar b & \a^3\epm   \right].
\end{align}
If we consider the rank 2 type in this case, we will induce the same contradiction as the hyperbolic case. So there is no rank 2 type simple element in  $ \Lambda^{\sigma}_{I} GL(3,\C)$ satisfies the $\tau_2$-reality condition.

\ms
\ms

When $g(\l) \in \Lambda^{\sigma,\tau}_{I} GL(3,\C)$ has a pole at $\a$ with $|\a|\neq 1$, ($\sigma, \tau$)-reality condition implies that it also has the same type of poles at $\{\e^{2}\a,\e^{4}\a, \bar{\a}^{-1},\e^{2}\bar{\a}^{-1},\e^{4}\bar{\a}^{-1}\}$. So we will simply try the product of two simple elements in $\Lambda^{\sigma}_{I} GL(3,\C)$ with poles at $\{\a,\e^{2}\a,\e^{4}\a \}$ and $\{\bar{\a}^{-1},\e^{2}\bar{\a}^{-1},\e^{4}\bar{\a}^{-1} \}$ respectively, i.e., we try $h(\l)=A_1 g_{\a.\ell_1}(\l) A_2 g_{\bar \a^{-1},\ell_2} (\l)$. Considering  the  Proposition \ref{p42} property (3), we have
\begin{align*}
h(\l)&=A_1 g_{\a.\ell_1}(\l) A_2 g_{\bar \a^{-1},\ell_2}(\l) \\
&=A_1A_2 \ g_{\a.\ell_1A_2}(\l) g_{\bar \a^{-1},\ell_2}(\l)\\
&= \tilde{A}_1 g_{\a.\tilde{\ell}_1}(\l) g_{\bar \a^{-1},\ell_2}(\l).
 \end{align*}
Without loss of generality, we set
  \begin{align} \label{4r}
  h(\l)=A g_{\a,\ell_1}(\l) g_{\bar \a^{-1},\ell_2} (\l)  = (h_{ij})_{1\leq i,j \leq3}
  \end{align}
here
\begin{align*}
&h_{11} =d \frac{ [\l^3(2b_1c_1-1)+\a^3][\l^3\bar \a^{3}(2b_2c_2-1)+1]+4 \l^3 |\a|^2 b_2c_1(2b_2c_2-1)(b_2c_1 + |\a|^2)}
{ (\l^3-\a^3)(\l^3\bar \a^{3}-1)(2b_1c_1-1)(2b_2c_2-1)}, \\
&h_{12} =d \frac{2 \l^2\bar \a^{2}c_2^2[\l^3(2b_1c_1-1)+\a^3]+(2b_2c_2-1)[2\l^2 \a c_1^2(2b_2c_2-1+\l^3\bar \a^{3})+4 \l^2 \a^2 \bar \a c_1c_2]}
{  (\l^3-\a^3)(\l^3\bar \a^{3}-1)(2b_1c_1-1)(2b_2c_2-1)}, \\
&h_{13} = d \frac{2 \l \bar \a c_2[\l^3(2b_1c_1-1)+\a^3]+(2b_2c_2-1)[4\l^4 \a \bar \a^2 b_2 c_1^2+2 \l \a^2 c_1(\l^3\bar \a^3+1)]}
{(\l^3-\a^3)(\l^3\bar \a^{3}-1)(2b_1c_1-1)(2b_2c_2-1)},\\
&h_{21} = \frac{2\l \a^2 b_1^2[\l^3\bar \a^3(2b_2c_2-1)+1]+(2b_2c_2-1)[2\l \bar \a b_2^2(2\a^3b_1c_1-\a^3+\l^3)+4\l^4 \a \bar \a^2 b_1b_2]}
{d (\l^3-\a^3)(\l^3\bar \a^{3}-1)(2b_2c_2-1)}, \\
&h_{22} = \frac{4\l^3 |\a|^4 b_1^2 c_2^2+(2b_2c_2-1)\{[\a^3(2b_1c_1-1)+\l^3](2b_2c_2-1+\l^3 \bar \a^3)+4\l^3\a \bar \a^2 b_1b_2\}}
{d (\l^3-\a^3)(\l^3\bar \a^{3}-1)(2b_2c_2-1)},\\
&h_{23} =  \frac{4\l^2 \a^2 \bar \a b_1^2 c_2+(2b_2c_2-1)[2\l^2 \bar \a^2 b_2(2\a^3b_1c_1-\a^3+\l^3)+2 \l^2 \a b_1(\l^3 \bar\a^3+1)]}
{d (\l^3-\a^3)(\l^3\bar \a^{3}-1)(2b_2c_2-1)}, \\
&h_{31} =\pm \frac{2 \l^2 \a b_1 [\l^3\bar\a^3(2b_2c_2-1)+1]+(2b_2c_2-1)[4\l^2 \a^2 \bar \a b_2^2c_1+2 \l^2 \bar \a^2 b_2(\l^3+\a^3)]}
{ (\l^3-\a^3)(\l^3\bar \a^{3}-1)(2b_2c_2-1)},\\
&h_{32} = \pm \frac{4\l^4 \a \bar \a^2b_1 c_2^2+(2b_2c_2-1)[2\l\a^2 c_1(\l^3 \bar \a^3+2b_2c_2-1)+2 \l \bar \a c_2(\l^3+\a^3)]}
{(\l^3-\a^3)(\l^3\bar \a^{3}-1)(2b_2c_2-1)}, \\
&h_{33} =\pm \frac{4\l^3 |\a|^2b_1 c_2+(2b_2c_2-1)[4\l^3 |\a|^4 b_2c_1+ (\l^3+\a^3)(\l^3 \bar\a^3+1)]}
{(\l^3-\a^3)(\l^3\bar \a^{3}-1)(2b_2c_2-1)},
\end{align*}
all these entries of $h(\l)$ are computed by Maple. It satisfies $\tau$-reality condition iff
 $$ \tau(h(\bar \l^{-1})) =h(\l),$$
 i.e.
 \beq \label{4s}
 \bar A \bar g_{\a,\ell_1}(\bar\l^{-1})   \bar g_{\bar \a^{-1},\ell_2} (\bar\l^{-1}) \ P_{12}\,T \   g_{\bar \a^{-1},\ell_2}^{t} (-\l) ^t g_{\a,\ell_1}^{t} (-\l) A^t =P_{12}\ T .
 \eeq

Step 1: Evaluate LHS of $\eqref{4s}$ at $\l=\infty$ to get
 $$ \bar A (I-3\bar \a^{-1} \diag(\bar R_1)) (I-3 \a \diag(\bar R_2)) \ P_{12}\,T   A^{t} =\ P_{12}\,T .\,$$
 This implies that
 \begin{align*}
 &\bpm \frac{-\bar d}{2\bar b_1 \bar c_1 -1} & &  \\ & \bar d^{-1} (1-2\bar b_1 \bar c_1) & \\ & & \mp 1 \epm
 \bpm\frac{-1}{2\bar b_2 \bar c_2 -1} & &  \\ &  (1-2\bar b_2 \bar c_2) & \\ & & -1\epm
 \ P_{12}\,T \ \bpm d & & \\ & d^{-1} & \\ & & \pm 1\epm
 \\
 & =\ P_{12}\,T \ .
 \end{align*}
 It follows that

 \begin{equation}\label{4w}
 |d|^2 =(2\bar b_1\bar c_1-1) (2\bar b_2\bar c_2 -1).
 \end{equation}

 \ms
Step 2: Compute the residue of LHS of $\eqref{4s}$ at $\l=-\a$ and $\l=\bar \a^{-1}$ to get
 $$  \overline{A g_{\a,\ell_1}(-\bar\a^{-1})g_{\bar \a^{-1},\ell_2} (-\bar\a^{-1})}\ P_{12}\,T \ g_{\bar\a^{-1},\ell_2}^{t}(\a)R_1^{t}  =0.$$
 Since $Ag_{\a,\ell_1} (-\bar \a^{-1})$ is invertible, we have
 \beq \label{4t}
  \overline{ g_{\bar \a^{-1}.\ell_2} (-\bar\a^{-1})} \ P_{12}\,T \ g_{\bar\a^{-1},\ell_2}^{t}(\a)R_1^{t}  =0.
 \eeq
 Residues of LHS of $\eqref{4s}$ at $\l=-\e^2\a,\ -\e^4 \a$ give the same condition $\eqref{4t}$.
 Compute the residue of LHS of $\eqref{4s}$ at $\l=-\bar\a^{-1}$ and $\l=\a$ to get
 $$ \overline{ A g_{\a,\ell_1} (-\a) g_{\bar \a^{-1},\ell_2} (-\a)} \ P_{12}\,T \ R_2^{t} g_{\a,\ell_1}^{t} (\bar\a^{-1}) =0.$$
 But $g_{\a,\ell_1} (\bar\a^{-1})$ is invertible, so
 \beq \label{4u}
  \overline{A g_{\a,\ell_1} (-\a)  g_{\bar \a^{-1},\ell_2} (-\a)} \ P_{12}\,T \ R_2^{t}=0.
 \eeq

 \ms
Step 3: Consider $\eqref{4u}$. \eqref{4d} implies
 $\ker A g_{\a,\ell_1} (-\a)$ is spanned by $(c_1,b_1,1)^t$. So \eqref{4u} implies that
 \begin{equation}\label{4u2}
  g_{\bar \a^{-1},\ell_2} (-\a) \ P_{12}\,T \ \bpm \bar b_2 \\ \bar c_2 \\ 1 \epm  \parallel  \bpm c_1 \\b_1 \\1 \epm,
 \end{equation}

 where $\parallel$ means two vectors are parallel.
 Tedious, but direct computation implies that
  \beq \label{4v}
  g_{\bar \a^{-1},\ell_2} (-\a) \bpm \bar b_2 \\ \bar c_2 \\ -H/2 \epm   \parallel  \bpm L_1 \\L_2 \\L_3 \epm,\eeq
  where
$$
 \bpm L_1 \\L_2 \\L_3 \epm = \bpm
 \bar b_2 -\frac{2 c_2}{(2b_2c_2-1)(1+|\a|^6)} (|b_2|^2 +|\a|^4 |c_2|^2 +\frac{H}{2}|\a|^2) \\   \bar c_2  -\frac{2 b_2}{1+|\a|^6} (-|\a|^2 |b_2|^2 +|c_2|^2 -\frac{H}{2} |\a|^4 ) \\ -\frac{H}{2}-\frac2{1+|\a|^6} (|\a|^4 |b_2|^2 -|\a|^2 |c_2|^2 -\frac{H}{2}) \epm. $$
 By $\eqref{4v}$, we have
 \begin{align}
  b_1 &=\frac{- b_2 (-|\a|^2 |b_2|^2 +|c_2|^2-\frac{H}{2}|\a|^4) +\frac12 (1+|\a|^6) \bar c_2}{-|\a|^4|b_2|^2 +|\a|^2 |c_2|^2 +\frac H4  (1-|\a|^6)},   \label{4x}\\
   c_1 &=\frac{- c_2 (|b_2|^2 +|\a|^4 |c_2|^2+\frac{H}{2}|\a|^2) +\frac12 (2b_2c_2 -1) (1+|\a|^6) \bar b_2}{(2b_2c_2 -1)[-|\a|^4|b_2|^2 +|\a|^2 |c_2|^2 +\frac H4  (1-|\a|^6)]}. \label{4y}
 \end{align}

 \ms
Step 4: Consider $\eqref{4t}$. $\sigma_1$-reality condition of $g_{\bar \a^{-1},\ell_2}(\l)$ implies
 $$ g_{\bar \a^{-1},\ell_2}(-\l) \ P_{12} \  g_{\bar \a^{-1},\ell_2}^t(\l)=P_{12}.$$
 Compute residue at $\l=\bar \a^{-1}$ to get
 $$ g_{\bar \a^{-1},\ell_2}(-\bar \a^{-1}) \ P_{12} \  R_2^t =0,$$
 which implies that
 $$\ker  g_{\bar \a^{-1},\ell_2}(-\bar \a^{-1}) =\Span_\C \bpm c_2 \\ b_2\\1 \epm.$$
 So \eqref{4t} implies that
 $$ P_{12}\  T g_{\bar \a^{-1},\ell_2}^{*} (\a) \bpm \bar b_1 \\ \bar c_1 \\ 1 \epm \parallel  \bpm b_1 \\ c_1 \\1 \epm. $$
 It is equivalent to $\eqref{4v}$.

Conversely, if the product of two simple element $h(\l)$ satisfies \eqref{4w}, \eqref{4x} and \eqref{4y}, $h(\l)$ satisfies  the $\tau$-reality condition.

We conclude the above computations into the following theorem.

\bthm \label{t43}
(1)Take any $\a\in \mathbf{S}^1 $ and  $b \in \mathbb{C}$ so that $|b|^2>\frac{1}{2}$. Then  \\
$c=\bar b $ and
$
|d|^2 =(2 | b|^2 -1)
$
  $\Leftrightarrow$ $A m_{\a,\ell}(\l) \in \Lambda^{\sigma,\tau_1}_{I} GL(3,\C) \Leftrightarrow A g_{\a,\ell}(\l) \in \Lambda^{\sigma,\tau_2}_{I} GL(3,\C)$ .\\
(2)Take any $\a\in \C^{\times}/\mathbf{S}^1 $, $b_2, c_2\in \C$, $2b_2c_2\neq1$, define
\begin{align*}
&|d|^2 =(2 b_1 c_1-1)(2 b_2c_2-1),  \\
 b_1 &=\frac{- b_2 (-|\a|^2 |b_2|^2 +|c_2|^2-\frac{H}{2}|\a|^4) +\frac12 (1+|\a|^6) \bar c_2}{-|\a|^4|b_2|^2 +|\a|^2 |c_2|^2 +\frac H4  (1-|\a|^6)},  \\
   c_1 &=\frac{- c_2 (|b_2|^2 +|\a|^4 |c_2|^2+\frac{H}{2}|\a|^2) +\frac12 (2b_2c_2 -1) (1+|\a|^6) \bar b_2}{(2b_2c_2 -1)[-|\a|^4|b_2|^2 +|\a|^2 |c_2|^2 +\frac H4  (1-|\a|^6)]},
   \end{align*}
\[
   A=\diag (d,d^{-1},1),~~
   \ell_1= \mathbb{C}\cdot ( b_1,  c_1,  1) ,~~\ell_2= \mathbb{C}\cdot ( b_2, c_2, 1 ).
\]
Then $h(\l):=Ag_{\a,\ell_1}(\l) g_{\bar \a^{-1},\ell_2} (\l) \in \Lambda^{\sigma,\tau}_{I} GL(3,\C)$ .

\ethm
\rem  It would be interesting to consider the converse question: whether any rational element with 6 poles must take the form of $h(\l)$? More generally, it would be interesting to know whether any rational element can be factorized into product of simple ones:
$A m_{\a,\ell}(\l)$(or $A g_{\a,\ell}(\l)$) and $h(\l)$? Such factorization problem has been treated by Uhlenbeck \cite{Uh89}  in the  case  of unitary group. We will answer these questions in a subsequent paper.

\rem  \label{r45}
The formula \eqref{4w} implies $(2 b_1 c_1-1)(2 b_2 c_2-1)$ decides $|d|^2$. Since $|d|^2$ is real positive, there is an important restriction that the choice of $\a,b_2,c_2$ must keep
\[
(2 b_1 c_1-1)(2 b_2 c_2-1)>0.
  \]
Substitute $b_1,\ c_1$ into $\eqref{4w}$, we get
\begin{align*}
|d|^2=(2b_1 c_1-1)(2b_2 c_2-1)=\frac{\Psi_{\a,b_2,c_2}}{(-|\a|^4|b_2|^2 +|\a|^2 |c_2|^2 +\frac H4  (1-|\a|^6))^2},
\end{align*}
here
\begin{align} \label{4z}
\Psi_{\a,b_2,c_2}=&\frac{1}{4}|2b_2c_2-1|^2|\a|^{12}-|c_{2}|^{4}|\a|^{10}-H|c_{2}|^{2}|\a|^{8}+(-2 Re(b_2c_2)-\frac{1}{2})|\a|^6  \notag \\
&-H|b_2|^{2}|\a|^{4}-|b_2|^{4}|\a|^{2}+\frac{1}{4}|2b_2c_2-1|^{2}.
\end{align}
It is easy to see $(2 b_1  c_1-1)(2 b_2 c_2-1)$ is real and only need to choose $\a,b_2,c_2$ such that $\Psi_{\a,b
_2,c_2}>0$, then $(2 b_1  c_1 -1)(2 b_2 c_2-1) > 0$. For convenience, we will always assume $\a,b_2,c_2$ satisfy this restriction when we choose a simple element $h(\l)$ afterward.

\section{Dressing actions of simple elements}
In this section we use the simple element to give the dressing actions on definite affine spheres.\par
Let us review the technique of dressing action (the idea of dressing was dated from \cite{Zak79} but see \cite{Gue79} or \cite{Te08} for an elementary introduction).
Let $G=SL(3,\C)$, $g(\l)\in \Lambda_{I}^{\tau,\sigma}G$, and $E(z,\bar z, \lambda)\in \Lambda_{\mathbb{C}^{\times}}^{\sigma.\tau}G$ is the frame of a definite affine sphere. Assume we can do the following factorization for each fixed $(z, \bar z)$:
\begin{equation} \label{5a}
g(\l)E(z,\bar z,\l)=\tilde E(z,\bar z,\l) \tilde g(z,\bar z,\l),
\end{equation}
with $\tilde E(z,\bar z,\l)\in \Lambda_{\mathbb{C}^{\times}}^{\tau,\sigma}G$, $\tilde g(z,\bar z,\l)\in \Lambda_{I}^{\tau,\sigma}G$. Then $\tilde E(z,\bar z,\l)$ will also be the frame of a new definite affine sphere. We sketch the proof here. By Theorem \ref{loop}, it suffices to prove that $\tilde E^{-1}\tilde E_{z}$ and $\tilde E^{-1}\tilde E_{\bar z}$ are linear in $\l$ and $\l^{-1}$ respectively.
But
\begin{align*}
\tilde E^{-1}\tilde E_{z}&=\tilde g(z,\bar z,\l) E^{-1}E_{z}\tilde g(z,\bar z,\l)^{-1}-\tilde g(z,\bar z,\l)_{z}\tilde g(z,\bar z,\l)^{-1}\\
&=\tilde g(z,\bar z,\l)(u_{0}+\l u_{1})\tilde g(z,\bar z,\l)^{-1}-\tilde g(z,\bar z,\l)_{z}\tilde g(z,\bar z,\l)^{-1}.
\end{align*}
On the left hand side $\tilde E^{-1}\tilde E_{z}$ is holomorphic in $\l \in \C^\times$; on the right it has a simple pole at $\infty$. So
\[
\tilde E^{-1}\tilde E_{z}=\tilde u_{0}+\l \tilde u_{1},
\]
where $\tilde u_{0}$, $\tilde u_{1}$ are independent of $\l$. Similarly, $\tilde E^{-1}\tilde E_{\bar z}$  is linear in $\frac{1}{\l}$. This completes the proof.

Furthermore, $g\ast E := \tilde{E}$ defines a group action of $\Lambda_{I}^{\tau,\sigma}G$ on the frames of the definite affine spheres, which is called the {\it dressing action\/}.

\rem
 We can consider the dressing action of $g \in \Lambda_{I}^{\sigma,\tau}GL(3,\mathbb{C})$ on $\Lambda_{\mathbb{C}^{\times}}^{\sigma,\tau}SL(3,\mathbb{C})$, since a scaling by $(\det g)^{-1/3}$ can make $g$ lies in $SL(3,\mathbb{C})$ by Proposition 4.2 and scaling does not affect the dressing action. What's more, the differences between rank 1 type $A g_{\a,\ell}(\l)$ and rank 2 type $A m_{\a,\ell}(\l)$, and between $A = diag(d,d^{-1},1)$ and  $diag(d,d^{-1},-1)$ are also scalings. Without loss of generality, we only consider one type of simple element and choose  $A = diag(d,d^{-1},1)$.

\rem \label{r52}
Since
\begin{equation}\label{5a2}
\Lambda_{I}^{\sigma,\tau}GL(3,\mathbb{C}) \cap \Lambda_{\mathbb{C}^{\times}}^{\sigma,\tau}SL(3,\mathbb{C}) = \{diag(e^{i\theta}, e^{-i\theta}, 1) ~|~\theta \in \mathbb{R}\},
\end{equation}
the factorization \eqref{5a} will not be unique and the ambiguity lies in \eqref{5a2}. By Lemma 3.2, we can eliminate the  ambiguity  by requiring  $d>0$ in the matrix $A$, without changing the geometric property of dressing action. We also know that
\[
\Lambda_{I}^{\sigma}GL(3,\mathbb{C}) \cap \Lambda_{\mathbb{C}^{\times}}^{\sigma}SL(3,\mathbb{C}) = \{diag(d,d^{-1}, 1)~|~  d \in \mathbb{C}^{\times}\}.
\]

\ms
\begin{lem} \label{l53}
Let $A g_{\a,\ell}(\l) \in \Lambda_{I}^{\sigma}GL(3,\mathbb{C})$ as in \eqref{4g}, and $E(\l) \in \Lambda_{\mathbb{C}^{\times}}^{\sigma}SL(3,\mathbb{C})$. If $\tilde{\ell}:= \ell E(\a)\nsubseteq\Delta$, then we have
\begin{align} \label{5d}
\tilde{E}(\l):=A g_{\a,\ell} \cdot E(\l) \cdot g_{\a,\tilde{\ell}}^{-1}\tilde{A}^{-1}
\end{align}
lies in $\Lambda_{\mathbb{C}^{\times}}^{\sigma}SL(3,\mathbb{C})$ for any $\tilde{A}=diag\{\tilde{d},\tilde{d}^{-1},1\}$, with arbitrary $\tilde{d}\in \mathbb{C}^{\times}$.
\end{lem}

\begin{proof}
Since $\tilde{E}(\l)$ satisfied the $\sigma$-reality condition and holomorphic in $\mathbb{C}^{\times}$ except for possible simple poles coming from the poles of $A g_{\a,\ell}$ and $g_{\a,\tilde{\ell}}^{-1}\tilde{A}^{-1}$, we only need to prove that residues of $\tilde{E}(\l)$ are zero at both $\a$ and $-\a$. But
\[
 \sigma_1(A g_{\a,\ell}(\l)) = A g_{\a,\ell}(-\l) \Longleftrightarrow P_{12} = (A g_{\a,\ell}(\l)) \; P_{12} \;A (g_{\a,\ell}(-\l))^t ,
\]
whose residue is zero at $\a$ implies that $\ell \; P_{12}\;(g_{\a,\ell}(-\a))^t=0$, or equivalently $(g_{\a,\ell}(-\a)) \;P_{12}\;\ell^t=0$. These two equations are also true for $\tilde{\ell}$. Therefore $(b_1,c_1,1)E(\a) \in \tilde{\ell}$ and the special form of R in \eqref{4f2} imply that
\[
Res_\a \tilde{E} = 2 \a \;R \;E(\a) \; P_{12} \; g_{\a,\tilde{\ell}}(-\a)^t \tilde{A}^t\; =0,
\]
and $E(-\a)\;P_{12}\;(\tilde{b}_1,\tilde{b}_2,1)^t \in [E(-\a)\;P_{12}\;E(\a)^t]\ell^t = P_{12}\;\ell^t$ implies that
\[
Res_{-\a} \tilde{E} = -2 \a\; A g_{\a,\ell}(-\a)\;E(-\a) \; P_{12} \; \tilde{R}^t \;P_{12}=0,
\]
The proof is completed once we notice that det$ \tilde {E} = 1$ by Proposition 4.2.
\end{proof}

When we impose  the $\tau$-reality condition in the following, the above $\tilde{d}$  will be uniquely determined by requiring positiveness  (see Remark \ref{r52}). Now we begin to compute the dressing actions of two types of  elements in Theorem 4.3.

For the first type $|\a| = 1$. We consider the simple element $A m_{\a,\ell}$ for hyperbolic case($\tau = \tau_1$) first. Let $E(\l) \in \Lambda_{\mathbb{C}^{\times}}^{\sigma,\tau_1}SL(3,\mathbb{C})$, by Lemma \ref{l53}, we get
\[
\tilde{E}(\l):=A m_{\a,\ell} \cdot E(\l) \cdot m_{\a,\tilde{\ell}}^{-1}\tilde{A}^{-1} \in  \Lambda_{\mathbb{C}^{\times}}^{\sigma}SL(3,\mathbb{C}).
\]
here $\tilde{\ell} = \ell E(-\a)$. Since $A m_{\a,\ell}$ satisfies $\tau$-reality condition, \eqref{4o} implies $\ell P_{12} = \bar \ell$. For $\tilde{\ell}$, we also have
\[
\tilde{\ell} P_{12} = \ell E(-\a) P_{12}  = \bar \ell \bar E(-\a) = \bar{\tilde{\ell}}.
\]
So we only need to choose $|\tilde{d}|^2 = 2|\tilde{b}|^2-1$ to get $\tilde{A} m_{\a,\tilde{\ell}_1}(\l) \in \Lambda_{I}^{\sigma, \tau_1}GL(3,\mathbb{C})$. Then we have $\tilde{E} \in \Lambda_{\mathbb{C}^{\times}}^{\sigma.\tau_1}SL(3,\mathbb{C})$.
This leads to the following Theorem:
\bthm \label{t54}
Given any hyperbolic  affine sphere $r$,  there is a family of hyperbolic  affine spheres $r_{\l}$ with the local affine frames $E(z,\bar z, \l) \in \Lambda_{\mathbb{C}^{\times}}^{\sigma.\tau_1}SL(3,\mathbb{C})$. We normalize this affine frames at (0,0), i.e., $E(0,0 ,\l)=I$.  Pick any simple element $A m_{\a,\ell} \in \Lambda_{I}^{\sigma, \tau_1}GL(3,\mathbb{C})$. If  $\ell E(z, \bar z ,-\a)\nsubseteq\Delta$ and $2|\tilde{b}|^2-1 > 0$ while $\tilde{b}$ is determined by $\tilde{\ell}:= \ell E(z,\bar z,-\a) \parallel (\tilde{b}, \bar{\tilde{b}},1)$. We define $\tilde{A}=diag\{\tilde{d},\tilde{d}^{-1},1\}$ and $\tilde{d}=\sqrt{2|\tilde{b}|^2-1}$, and get the factorization:
\[
A m_{\a,\ell}  \cdot E(z,\bar z, \l) =\tilde{E}(z,\bar z,\l) \cdot \tilde{A}(z,\bar z ) m_{\a,\tilde{\ell}(z,\bar z)} \in
\Lambda_{\mathbb{C}^{\times}}^{\sigma.\tau_1}SL(3,\mathbb{C}) \times \Lambda_{I}^{\sigma, \tau_1}GL(3,\mathbb{C}).
\]
From the new affine frame $\tilde{E}(z,\bar z,\l)$, we can get explicit formula of new definite affine sphere:

\begin{equation}\label{5d2}
\hat{r} =\frac{-H(\l^3+\a^3)re^{\psi}-4\a^3(ln\phi)_{\bar z}r_z-4\l^3(ln\phi)_zr_{\bar z}}{(\l^3+\a^3)e^{\psi}},
\end{equation}
where $\phi:=(b,\; \bar b,\; 1)r_{-\a} $ is a scale solution of \eqref{eqla} with parameter $-\a$. This new definite affine sphere has new affine metric $\tilde{d}^2 \cdot e^\psi$ and the same affine cubic form (and the same affine mean curvature).

For the elliptic affine sphere, we have the similar results. Only need to change $-\a$ to $\a$ and a scaling for \eqref{5d2} by $\frac{\l^3-\a^3}{\l^3+\a^3}$.
\ethm

\begin{proof}
Since $\ell E(z,\bar z,-\a) \parallel (\tilde{b}, \bar{\tilde{b}},1)$, we get
\begin{align*}
&(b,\; \bar b, \;1) \;\;  \bpm \frac{1}{-\a}\sqrt{-2H}e^{\frac{-\psi}{2}}(r_{-\a})_{z},&-\a \sqrt{-2H}e^{\frac{-\psi}{2}}(r_{-\a})_{\bar z},&-H r_{-\a}\epm \notag \\
=&\bpm \frac{1}{-\a}\sqrt{-2H}e^{\frac{-\psi}{2}}\phi_{z},&-\a \sqrt{-2H}e^{\frac{-\psi}{2}}\phi_{\bar z},&-H \phi\epm
\end{align*}
where $\phi:=(b,\; \bar b,\; 1)r_{-\a} $ is a scale solution of \eqref{eqla} with parameter $-\a$.
The third column of $\tilde{E}(z,\bar z,\l)$ gives new hyperbolic affine sphere, so does the affine transformation of it by $m_{\a,\ell}(\l)^{-1} \; A^{-1}$:
\begin{align*}
\hat{r}&=m_{\a,\ell}(\l)^{-1} \; A^{-1} \;(\tilde{E}(\l))_3 \notag\\
&= (E(\l) \; m_{\a,\tilde{\ell}}(\l)^{-1} \; \tilde{A}^{-1})_3 \notag\\
&= \bpm \frac{1}{\l}\sqrt{-2H}e^{\frac{-\psi}{2}}r_{z},&\l \sqrt{-2H}e^{\frac{-\psi}{2}}r_{\bar z},&-H r\epm
\cdot \frac{1}{\l^3+\a^3} \cdot \bpm 2\a^3\l\frac{\sqrt{-2H}e^{-\frac{\psi}{2}}\phi_{\bar z}}{H \phi} \\2\l^2\frac{\sqrt{-2H}e^{-\frac{\psi}{2}}\phi_{ z}}{H \phi} \\ \l^3+\a^3 \epm \notag\\
&=\frac{-H(\l^3+\a^3)re^{\psi}-4\a^3(ln\phi)_{\bar z}r_z-4\l^3(ln\phi)_zr_{\bar z}}{(\l^3+\a^3)e^{\psi}}.
\end{align*}
Recall the discussion in the beginning of this section, $\tilde E^{-1}\tilde E_{z}$ be expressed as:
\begin{align*}
\tilde E^{-1}\tilde E_{z}&=\tilde{A}  m_{\a,\tilde{\ell} } (u_{0}+\l u_{1}) (\tilde{A}  m_{\a,\tilde{\ell} })^{-1}+(\tilde{A}  m_{\a,\tilde{\ell} })_{z}(\tilde{A}  m_{\a,\tilde{\ell} })^{-1}. \\
&=\tilde u_{0}+\l \tilde u_{1}.
\end{align*}
We can get $\tilde u_1$ by
\begin{align*}
\tilde u_1
&=\lim_{\l\rightarrow \infty}\frac{\tilde E^{-1}\tilde E_{z}}{\l} \\
&=\lim_{\l\rightarrow \infty}\tilde{A}  m_{\a,\tilde{\ell} }\  u_{1}\ (\tilde{A}  m_{\a,\tilde{\ell} })^{-1}\\
&= \tilde A \ u_1 \ P_{12}\tilde A P_{12}\\
&= \bpm 0 & 0 & \frac{\sqrt{-2H}}{2} \tilde{d} h^{\frac 12} \\ U \tilde{d}^{-2}h^{-1} & 0 & 0 \\ 0 & \frac{\sqrt{-2H}}{2} \tilde{d}h^{\frac 12} & 0 \epm.
\end{align*}
Compare $\tilde u_1$ with the coefficient matrix of $\lambda$ in \eqref{eqla}, it is easy to see that this new definite affine sphere has new affine metric $\tilde{d}^2 \cdot e^\psi$ and the same affine cubic form (and the same affine mean curvature).
\end{proof}
The formula \eqref{5d2} is similar to the classical Tzitz\'{e}ica transformation of indefinite affine spheres.

Now we consider the dressing action of the second type $h(\l)=A g_{\a,\ell_1}(\l)  g_{\bar{\a}^{-1},\ell_2}(\l) \in \Lambda_{I}^{\sigma, \tau}GL(3,\mathbb{C})$  in Theorem \ref{t43} on $E \in \Lambda_{\mathbb{C}^{\times}}^{\sigma.\tau}SL(3,\mathbb{C})$. We just need to apply Lemma \ref{l53} twice:
\[
\hat{E}(\l):= g_{\bar{\a}^{-1},\ell_2}(\l) \cdot E(\l) \cdot g^{-1}_{\bar{\a}^{-1},\tilde{\ell}_2}(\l) \in \Lambda_{\mathbb{C}^{\times}}^\s SL_3\C
\]
\[
\tilde{E} (\l):=A g_{\a,\ell_1}(\l)g_{\bar{\a}^{-1},\ell_2}(\l) \cdot E(\l) \cdot g^{-1}_{\bar{\a}^{-1},\tilde{\ell}_2}(\l) g^{-1}_{\a,\tilde{\ell}_1}(\l) \tilde{A}^{-1}\in \Lambda_{\mathbb{C}^{\times}}^\s SL_3\C\\
\]
where
$$
\ell_1 = \C \bpm  b_1 \\  c_1\\ 1 \epm^t, \ell_2=\C \bpm  b_2 \\  c_2 \\ 1 \epm^t,
$$
$$
A=diag(d,d^{-1},1),
$$
and $\tilde{\ell}_1,\tilde{\ell}_2$ satisfying
\begin{align*}
\tilde{\ell}_2 = \ell_2 E(\bar{\a}^{-1}), \qquad\qquad\qquad\quad\\
\tilde{\ell}_1 = \ell_1 \hat{E}(\a) = \ell_1 g_{\bar \a^{-1}, \ell_{2}}(\a) E(\a) g^{-1}_{\bar \a^{-1}, \tilde{\ell}_{2}}(\a),
\end{align*}
or
\beq \label{5e}
\ell_2 = \tilde{\ell}_2 E^{-1}(\bar \a^{-1}),
\eeq
\beq \label{5f}
\ell_1 = \tilde{\ell}_1  g_{\bar \a^{-1}, \tilde{\ell}_2}(\a)E^{-1}(\a)g^{-1}_{\bar \a^{-1},\ell_2}(\a).
\eeq
Since $h(\l)$ satisfies $\tau$-reality condition, from \eqref{4u2}, we have
$$
\bar \ell_2  \ P_{12}\ T \ g^t_{\bar \a^{-1}, \ell_2}(-\a)=\ell_1 P_{12},
$$
or
\beq \label{5g}
\ell_1 g_{\bar \a^{-1},\ell_2}(\a) = \bar \ell_2 T.
\eeq
Substitute $\ell_1$, $\ell_2$ in \eqref{5e}-\eqref{5f} into \eqref{5g}, and we get
\beq \label{5h}
\tilde{\ell}_1  g_{\bar \a^{-1},\tilde{\ell}_2}(\a) = \bar{\tilde{\ell}}_2 T.
\eeq

Compare \eqref{5g} with \eqref{5h}, we can find that the relationship between $\tilde{\ell}_1$ and $\tilde{\ell}_2$ satisfies \eqref{4x}-\eqref{4y}. Only need to choose $\tilde{d}$ satisfies \eqref{4w}, we get $\tilde{h}(\l)=\tilde{A} g_{\a,\tilde{\ell}_1}(\l) g_{\bar{\a}^{-1},\tilde{\ell}_2}(\l) \in \Lambda_{I}^{\sigma, \tau}GL(3,\mathbb{C})$. Then we have $\tilde{E} \in \Lambda_{\mathbb{C}^{\times}}^{\sigma.\tau}SL(3,\mathbb{C})$.

\ms
  Summarizing the above computations, we have the following theorem.
  \bthm \label{t55}
  Given any definite affine sphere $r$,  there is a family of definite affine spheres $r_{\l}$ with the local affine frames $E(z,\bar z, \l) \in \Lambda_{\mathbb{C}^{\times}}^{\sigma.\tau}SL(3,\mathbb{C})$ . We normalize this affine frames at (0,0), i.e., $E(0,0 ,\l)=I$. Pick $h(\l)=A g_{\a,\ell_1}(\l) g_{\bar \a^{-1},\ell_2}(\l)\in\Lambda^{\sigma,\tau}_{I} GL(3,\C)$. If both $\tilde{\ell}_2: = \ell_2 E(\bar{\a}^{-1}) \parallel (\tilde{b}_2,\tilde{c}_2,1)$ and $\tilde{\ell}_1 = \ell_1 g_{\bar \a^{-1},  \ell_{2}}(\a) E(\a) g^{-1}_{\bar \a^{-1}, \tilde{\ell}_{2}}(\a) \parallel (\tilde{b}_1,\tilde{c}_1,1)$ are not in $\Delta$ and $(2 \tilde{b}_1 \tilde{c}_1-1)(2 \tilde{b}_2 \tilde{c}_2-1)>0$, we define $\tilde{A}=diag\{\tilde{d},\tilde{d}^{-1},1\}$ and $\tilde{d}=\sqrt{(2 \tilde{b}_1 \tilde{c}_1-1)(2 \tilde{b}_2 \tilde{c}_2-1)}$, and get the factorization:
\begin{equation} \label{5k}
 A g_{\a,\ell_1}(\l) g_{\bar{\a}^{-1},\ell_2}(\l) \cdot E(z,\bar z,\l)
=   \tilde{E} (z,\bar z,\l)\cdot \tilde{A}(z,\bar z) g_{\a,\tilde{\ell}_1(z,\bar z)}(\l)g_{\bar{\a}^{-1},\tilde{\ell}_2(z,\bar z)}(\l),
\end{equation}
here $\tilde{E} (z,\bar z,\l) \in \Lambda_{\mathbb{C}^{\times}}^{\sigma, \tau} SL(3,\C)$, $\tilde{A}(z,\bar z) g_{\a,\tilde{\ell}_1(z,\bar z)}(\l)g_{\bar{\a}^{-1},\tilde{\ell}_2(z,\bar z)}(\l) \in \Lambda_{I}^{\sigma, \tau} GL(3,\C) $.

From the new affine frame $\tilde{E}(z,\bar z,\l)$, we can get explicit formula of new definite affine sphere
\begin{align*}
\tilde{r}&=N\frac{\a\l^3 \tilde{b}_1 \ \{-2\sqrt{-2H}h^{-\frac 12}[\l^3 \bar{\a}^3 (2\tilde{b}_2\tilde{c}_2-1)-1]\ r_{\bar z} +
   4\sqrt{-2H}h^{-\frac 12} \bar{\a}^2 \tilde{c}_2^2 \ r_z +
   4H\bar \a \tilde{c}_2  \ r \}}
   {(\l^3 \bar{\a}^3+1)(2\tilde{b}_2\tilde{c}_2-1)(\l^3+\a^3)}\\
   &+N\frac{\a^2 \tilde{c}_2 \ [4\sqrt{-2H}h^{-\frac 12} \bar{\a} \l^3 \tilde{b}_2^2 \  r_{\bar z}
   +2\sqrt{-2H}h^{-\frac 12} (\l^3 \bar{\a}^3+1-2 \tilde{b}_2\tilde{c}_2) \ r_z
   +4H\bar{\a}^2\l^3\tilde{b}_2 \ r]}
   {(\l^3 \bar{\a}^3+1)(\l^3+\a^3)}\\
   &+N\frac{(-\l^3+\a^3)[2\sqrt{-2H}h^{-\frac 12} \bar{\a}^2 \l^3 \tilde{b}_2 \ r_{\bar z}
    -2\sqrt{-2H}h^{-\frac 12} \bar{\a} \tilde{c}_2 \ r_z
    +H (\l^3 \bar{\a}^3-1) \ r]}
   {(\l^3 \bar{\a}^3+1)(\l^3+\a^3)}
\end{align*}
here
\[
 N=\begin{pmatrix} \frac{1}{\sqrt{2}} & \frac{1}{\sqrt{2}} & 0 \\ \frac{\mathrm{i}}{\sqrt{2}} & \frac{-\mathrm{i}}{\sqrt{2}}& 0 \\ 0 & 0 & \sqrt{H/2} \end{pmatrix}
 \]
This new definite affine sphere has new affine metric $\tilde{d}^2 \cdot e^\psi$ and the same affine cubic form (and the same affine mean curvature).

\ethm
\begin{proof}
By Remark \ref{rem-reframe}, we have
\begin{align*}
\hat{r}&=( N [A g_{\a,\ell_1}(\l)g_{\bar{\a}^{-1},\ell_2}(\l)]^{-1} \tilde{E}(\l) N^{-1})_3 \\
         &= (N E(\l) g_{\bar{\a}^{-1},\tilde{\ell}_2}(\l)^{-1} g_{\a,\tilde{\ell}_1}(\l)^{-1}\tilde{A}_1^{-1})_3 \\
         &= (N E(\l) \ P_{12} \ g_{\bar{\a}^{-1},\tilde{\ell}_2}(-\l)^{t} g_{\a,\tilde{\ell}_1}(-\l)^{t} \tilde{A}_1^t)_3 \\
         &= N \,(\l\sqrt{-2H}h^{-\frac 12}r_{\bar z}, \frac{1}{\l}\sqrt{-2H}h^{-\frac 12}r_z, -Hr) \\
        & \quad \cdot
          \left( \begin {array}{ccc}
           \frac{\l^3 \bar{\a}^3 (2\tilde{c}_1\tilde{c}_2-1)-1}{(\l^3 \bar{\a}^3+1)(2\tilde{c}_1\tilde{c}_2-1)}&
           \frac {2 \bar \a \l \tilde{c}_1^2}{\l^3 \bar{\a}^3+1}&
           \frac {-2 \bar{\a}^2 \l^2 \tilde{c}_1}{\l^3 \bar{\a}^3+1} \\
           \noalign{\medskip}
           \frac {-2 \bar{\a}^2 \l^2 \tilde{c}_2^2}{(\l^3 \bar{\a}^3+1)(2\tilde{c}_1\tilde{c}_2-1)}&
           \frac {\l^3 \bar{\a}^3+1-2 \tilde{c}_1\tilde{c}_2}{\l^3 \bar{\a}^3+1}&
           \frac {2\bar{\a}\l\tilde{c}_2}{\l^3 \bar{\a}^3+1}\\
           \noalign{\medskip}
           \frac {2 \bar{\a} \l \tilde{c}_2}{(\l^3 \bar{\a}^3+1)(2\tilde{c}_1\tilde{c}_2-1)}&
           \frac {-2\bar{\a}^2\l^2\tilde{c}_1}{\l^3 \bar{\a}^3+1}&
           \frac {\l^3 \bar{\a}^3-1}{\l^3 \bar{\a}^3+1}
           \end {array}\right)
 \cdot
  \left( \begin {array}{c}
  \frac{-2 \a \l^2 \tilde{b}_1}{\l^3+\a^3} \\
  \noalign{\medskip}
  \frac{2 \a^2  \l \tilde{b}_2}{\l^3+\a^3} \\
  \noalign{\medskip}
  \frac{\l^3-\a^3}{\l^3+\a^3}
  \end {array}\right) \\
   &=N\frac{\a\l^3 \tilde{b}_1 \ \{-2\sqrt{-2H}h^{-\frac 12}[\l^3 \bar{\a}^3 (2\tilde{c}_1\tilde{c}_2-1)-1]\ r_{\bar z} +
   4\sqrt{-2H}h^{-\frac 12} \bar{\a}^2 \tilde{c}_2^2 \ r_z +
   4H\bar \a \tilde{c}_2  \ r \}}
   {(\l^3 \bar{\a}^3+1)(2\tilde{c}_1\tilde{c}_2-1)(\l^3+\a^3)}\\
   &+N\frac{\a^2 \tilde{b}_2 \ [4\sqrt{-2H}h^{-\frac 12} \bar{\a} \l^3 \tilde{c}_1^2 \  r_{\bar z}
   +2\sqrt{-2H}h^{-\frac 12} (\l^3 \bar{\a}^3+1-2 \tilde{c}_1\tilde{c}_2) \ r_z
   +4H\bar{\a}^2\l^3\tilde{c}_1 \ r]}
   {(\l^3 \bar{\a}^3+1)(\l^3+\a^3)}\\
   &+N\frac{(-\l^3+\a^3)[2\sqrt{-2H}h^{-\frac 12} \bar{\a}^2 \l^3 \tilde{c}_1 \ r_{\bar z}
    -2\sqrt{-2H}h^{-\frac 12} \bar{\a} \tilde{c}_2 \ r_z
    +H (\l^3 \bar{\a}^3-1) \ r]}
   {(\l^3 \bar{\a}^3+1)(\l^3+\a^3)}.
\end{align*}
\end{proof}

\section {Basic examples}
In this section, we will apply our results to construct some basic examples.
\ms

\noindent\textbf{Example 6.1} (The \textbf{vacuum} solution).
Assume $H =-2$ and $U =1$ in the equation \eqref{eqgc} for hyperbolic definite affine
spheres. Then it admits the trivial solution $\psi=0$ (also called the vacuum solution), and the corresponding
surface is $x_1x_2x_3 = \frac{\sqrt{3}}{72}$. One can integrate \eqref{eqla} to obtain the whole family
of frames. Note that the surface is independent of the parameter $\l$. So the
associated family is really a family of parameterizations of the same affine
sphere.

We may choose the following conformal parametrization of the
vacuum definite affine sphere after certain affine transformation:
$$X =\frac{\sqrt{3}}{6} \bpm R(\l) \\ R(\epsilon^{2}\l) \\  R(\epsilon^{4}\l)\epm,$$
where
  $$\epsilon=e^{\frac{\pi i}{3}}, \;R(\l)=\exp(\l z+\frac{1}{\l}\bar z).$$
Then
\begin{align*}
F(z,\bar z,\l)=&\bpm \frac{1}{\l}\sqrt{-2H}e^{\frac{-\psi}{2}}X_{z},&\l\sqrt{-2H}e^{\frac{-\psi}{2}}X_{\bar z},&-HX\epm\\
=\frac{\sqrt{3}}{3} & \bpm R(\l) & R(\l) & R(\l)\\  \epsilon^{2}R(\epsilon^{2}\l) & \epsilon^{4}R(\epsilon^{2}\l) & R(\epsilon^{2}\l)\\ \epsilon^{4}R(\epsilon^{4}\l) & \epsilon^{2}R(\epsilon^{4}\l)& R(\epsilon^{4}\l)\epm
\end{align*}
Normalize this frame in a neighborhood of 0, we get
\begin{align} \label{6a1}
E(z,\bar z,\l):=F^{-1}(0,0,\l)F(z,\bar z,\l) =  exp(\l P_{132}\cdot z+\frac{1}{\l} P^t_{132} \cdot \bar z),
\end{align}
where
 \begin{align*}
P_{132} = \bpm 0 & 0 &1 \\  1& 0 & 0\\ 0 & 1 & 0\epm .
 \end{align*}

\noindent\textbf{Example 6.2} (The one `soliton' solution).
Choose simple element $A g_{\a,\ell}$ with $\a=i, \ell=\C(-\frac{1}{2}+i,-\frac{1}{2}-i,1)$, and $E(z,\bar z,\l)$ is the normalized frame of vacuum definite affine sphere in \eqref{6a1}.
Then we have the factorization:
 \begin{align*}
A g_{\a,\ell}(\l)E(\l)=\tilde{E}(\l)\tilde A g_{\a,\tilde \ell}(\l),
\end{align*}
where
\begin{align*}
\tilde \ell=\ell E(\a)=\C(\tilde b,\bar{\tilde b}, 1),
\end{align*}
\begin{align*}
\tilde{b}=\frac{(-\frac 34-\frac{\sqrt{3}}{2}+\frac 32 i +\frac {3\sqrt{3}}{4} i)e^{y-\sqrt{3}x}+(-\frac 34+\frac{\sqrt{3}}{2}+\frac 32 i -\frac {3\sqrt{3}}{4} i)e^{y+\sqrt{3}x}}
{(\frac 32-\sqrt{3})e^{y-\sqrt{3}x}+(\frac 32+\sqrt{3})e^{y+\sqrt{3}x}},
\end{align*}
\begin{align*}
\tilde{d}=\sqrt{2 |\tilde{b}|^2-1}=\sqrt{-\frac{3 [(4\sqrt{3}-7)e^{4\sqrt{3}x}- 4e^{2\sqrt{3}x}-4\sqrt{3}-7]}
{[(2\sqrt{3}-3)e^{2\sqrt{3}x}- (2\sqrt{3}+3)]^2}}.
\end{align*}
By Theorem \ref{t54}, we can get the solution of Tzitz\'{e}ica equation \eqref{eqgc} (new affine metric):
 \begin{align} \label{6b}
 h=e^{\tilde \psi}=\tilde d^2
 =\frac{3 [(7-4\sqrt{3})e^{4\sqrt{3}x}+ 4e^{2\sqrt{3}x}+4\sqrt{3}+7]}
{[(2\sqrt{3}-3)e^{2\sqrt{3}x}- (2\sqrt{3}+3)]^2}.
  \end{align}
Recall the one soliton type solution of Tzitz\'{e}ica equation in \cite{Kap97}:
\begin{equation}\label{6c}
h = 1 - 2 (ln \tau )_{z \bar z}, ~with ~ \tau = 1- e^{k z+\frac{3\bar z}{k}+s}.
\end{equation}
Here we adopt the usual notation $\tau$ for $\tau$-function, not the same as the reality condition before.
After some computation, we can see our solution \eqref{6b} is coincide with \eqref{6c} with $k = \sqrt{3}$ and $ s = \ln(\frac{2\sqrt{3}-3}{3+2\sqrt{3}})$. We give the picture new solution in Figure 1.

Recall a complete hyperbolic affine sphere in \cite{Roland Hildebrand.{2013}}:
\begin{align} \label{6cas}
 \mathbb{R}^2_{++} \ni \bpm x \\  y \epm \rightarrow  \vec{r} = \bpm r_1 \\  r_2 \\ r_3 \epm
 = \bpm -y(\frac{3}{2}\ln(x)+3\ln(y)+x) \\x^{-\frac{3}{2}}y^{-2}\sqrt{1+x} \\ y \epm,
\end{align}
which is asymptotic to the boundary of the cone obtained by the homogenization of the epigraph of the exponential function. To compare with this hyperbolic affine sphere, we use the conformal parametrization in \cite{lin Wang}:
\begin{align*}
\vec{r'} = \bpm \frac{1}{\sqrt{3}}\sinh^{-1}( \sqrt{3}x) (\sinh^2( \sqrt{3} x)+3 y) e^{y}\\
 -\frac{1}{\sqrt{3}}\sinh^{-1}( \sqrt{3} x) \sqrt{1+\sinh^2(\sqrt{3} x)} e^{-2y}\\
   -\frac{1}{\sqrt{3}}\sinh^{-1}(\sqrt{3} x) e^{y} \epm.
\end{align*}
with the affine metric, affine mean curvature and affine cubic form are:

\begin{equation}\label{6inv}
g = \frac{3}{2}(\mathrm{csch}( \sqrt{3}x )^2+\frac{2}{3})(dz~ d\bar z),~~~ H = -2, ~~~ C = I\cdot (dz^3+d \bar z^3).
\end{equation}

We only need to make a  motion of variable $x$ by : $x \rightarrow x- \frac{s}{2\sqrt{3}}$ and choose $\l = i$ in our case. Then we get an definite affine sphere with the same invariants with  \eqref{6cas}, i.e. one of the family of definite affine sphere  \eqref{5d2} is a complete hyperbolic affine sphere. Since the specific formula is too long to place in this paper, we give the picture of this affine sphere instead.

 \begin{figure}[!htb]
\begin{center}
{
\includegraphics[width=0.5 \columnwidth]{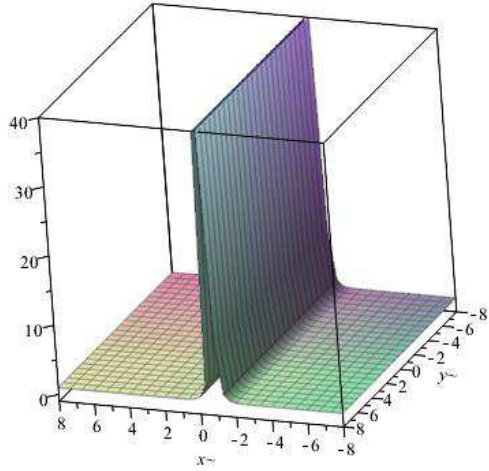}
}
\caption{$-8<x<8,-8<y<8$}
\end{center}
\end{figure}
 \begin{figure}[!htb]
\begin{center}
{
\includegraphics[width=0.5 \columnwidth]{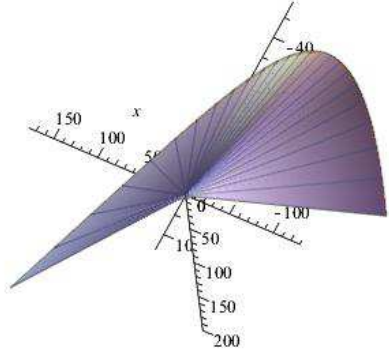}
}
\caption{$-2<x<2,-2<y<2$}
\end{center}
\end{figure}
\noindent\textbf{Example 6.3} (The two `soliton' solution). There are two cases.\\
(1) Choose simple element $A m_{\a,\ell}$, where $\a=i, \ell=\C(1+i,1-i,1)$, and $E(z,\bar z,\l)$ is \eqref{6a1}. Then we have the factorization
 \begin{align*}
A m_{\a,\ell}(\l)E(\l)=\tilde{E}(\l)\tilde A m_{\a,\tilde \ell}(\l),
\end{align*}
where
\begin{align*}
\tilde \ell=\ell E(-\a)=\C(\tilde b,\bar{\tilde b}, 1),
\end{align*}
\begin{align*}
\tilde{b}=\frac{(-3i+\sqrt{3})e^{\sqrt{3}x-y}+(-3i-\sqrt{3})e^{-\sqrt{3}x-y}-6e^{-2y}}
{-2(\sqrt{3}e^{\sqrt{3}x-y}-\sqrt{3}e^{-\sqrt{3}x-y}+3e^{2y})},
\end{align*}
\begin{align*}
\tilde{d}=\sqrt{2 |\tilde{b}|^2-1}=\sqrt{\frac{3(3e^{4y}-4e^{-2y}-4\sqrt{3}e^{\sqrt{3}x+y}+4\sqrt{3}e^{-\sqrt{3}x+y}+e^{2\sqrt{3}x-2y}+e^{-2\sqrt{3}x-2y})}
{(\sqrt{3}e^{\sqrt{3}x-y}-\sqrt{3}e^{-\sqrt{3}x-y}+3e^{2y})^2}}.
\end{align*}
By Theorem \ref{t54}, we can get the solution of Tzitz\'{e}ica equation \eqref{eqgc} (new affine metric):
 \begin{align*}
 &h=e^{\tilde \psi}=\tilde d^2\\
 &=\frac{3(3e^{4y}-4e^{-2y}-4\sqrt{3}e^{\sqrt{3}x+y}+4\sqrt{3}e^{-\sqrt{3}x+y}+e^{2\sqrt{3}x-2y}+e^{-2\sqrt{3}x-2y})}
{(\sqrt{3}e^{\sqrt{3}x-y}-\sqrt{3}e^{-\sqrt{3}x-y}+3e^{2y})^2}.
  \end{align*}
Recall the  two soliton type solution of Tzitz\'{e}ica equation in \cite{Kap97}:
\begin{equation*}
h = 1 - 2 (ln \tau )_{z \bar z},
\end{equation*}
with
\begin{equation}\label{6d}
\tau = 1+ e^{k_1z+\frac{3\bar z}{k_1}+s_1}+e^{k_2z+\frac{3\bar z}{k_2}+s_2}+\frac{(k_1-k_2)^2(k_1^2-k_1k_2+k_2^2)}{(k_1+k_2)^2(k_1^2+k_1k_2+k_2^2)} e^{(k_1+k_2)z+(\frac{3}{k_1}+\frac{3}{k_2})\bar z+s_1+s_2}
\end{equation}
Then we get the $\tau$-function of the new solution $h$:
\[
\tau = 1 + \frac{\sqrt{3}}{3} e^{\sqrt{3}x-3y} - \frac{\sqrt{3}}{3} e^{-\sqrt{3}x-3y}.
\]
It is coincide with the $\tau$-functions \eqref{6d} with parameters
$k_1=\frac{\sqrt{3}-3 i}{2}$, $k_2=\frac{-\sqrt{3}-3 i}{2}$, $s_1=ln(\frac{\sqrt{3}}{3})$, $s_2=ln(-\frac{\sqrt{3}}{3})$.

We give the picture of this solution in Figure 3. Notice in Figure 3, the solution goes below the red plane (z=0) only in the very small region, which means the metric become negative in a curve. It is easy to see that Tzitz\'{e}ica equation \eqref{eqgc} equals to
\beq \label{6h}
h_{z \bar z}h- h_z h_{\bar z} - h^3 +1=0.
\eeq
and the negative parts are also the solution of \eqref{6h}. From geometric point of view, if we fix the affine mean curvature $H=-2$, h is positive means the affine metric is positive definite, then it gives a local hyperbolic affine sphere; h is negative means the affine metric is negative definite, it gives a local elliptic affine sphere.

So the solution  $h=e^{\tilde \psi}$ can be extended to be a global solution of  Tzitz\'{e}ica equation \eqref{6b} and  we get a global mixed definite affine sphere. Some parts of it are hyperbolic type, the other are elliptic.

 We also use the formula \eqref{5d2} to get the new definite affine sphere:
\[
\tilde{X}=\frac{-H(\l^3+\a^3)Xe^{\psi}-4\a^3(ln\phi)_{\bar z}X_z-4\l^3(ln\phi)_zX_{\bar z}}{(\l^3+\a^3)e^{\psi}}
\]
here
\[
\phi=\ell F^{-1}_{\l}(0,0) X(-\a)=\frac 12 \,{{\rm e}^{2\,y}}+\frac 16\,{{\rm e}^{-\frac 13\,\sqrt {3} \left( \sqrt {3}y
-3\,x \right) }}\sqrt {3}-\frac 16\,{{\rm e}^{-\frac 13\,\sqrt {3} \left( \sqrt
{3}y+3\,x \right) }}\sqrt {3}
\]
Since the specific formula of new definite affine sphere is too long to place in this paper, we give the picture of it in Figure 4 with $\l = e^{\frac{2}{3} \pi i}$ instead.

 \begin{figure}[!htb]
\begin{center}
{
\includegraphics[width=0.5 \columnwidth]{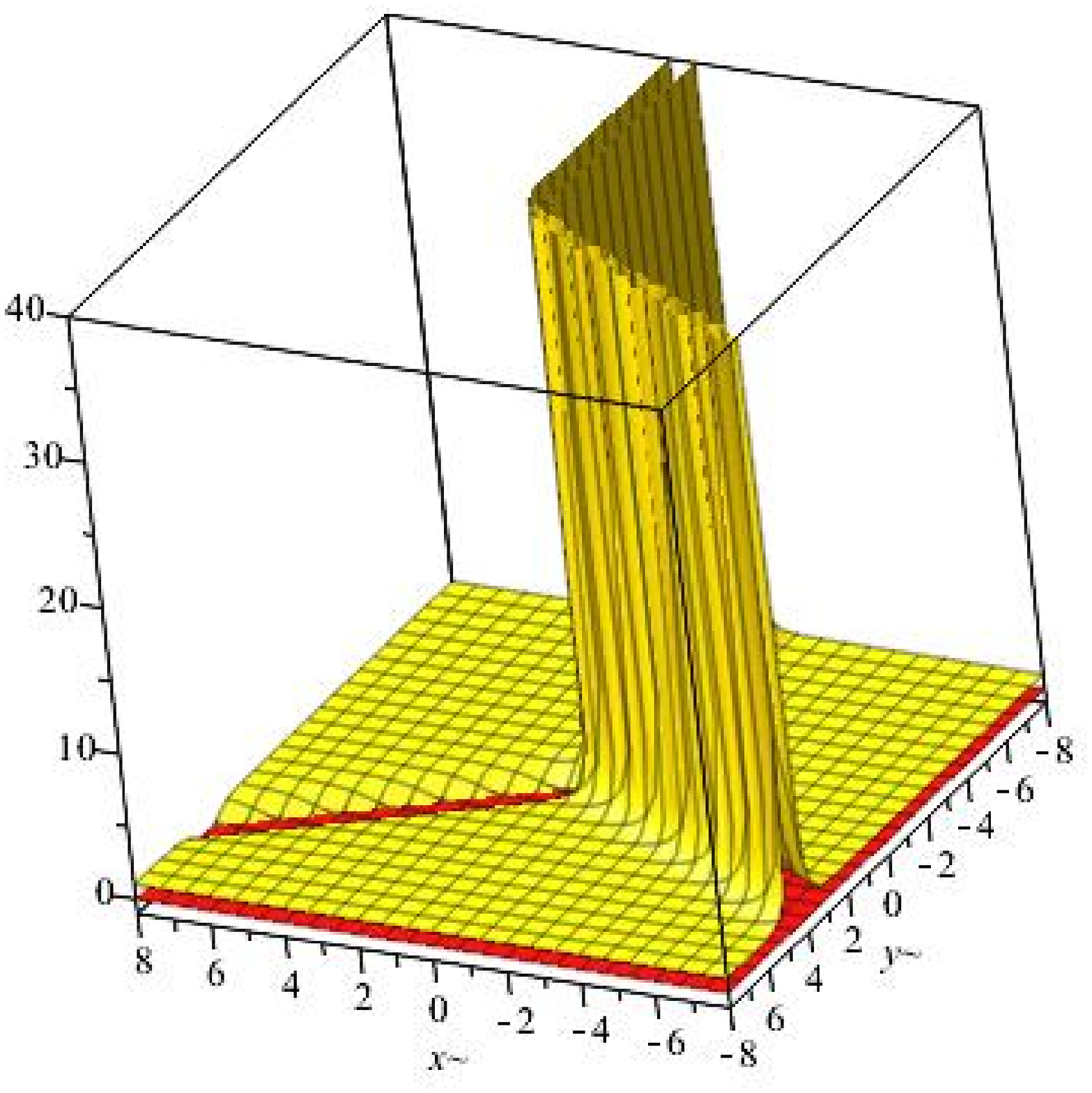}
}
\caption{$-8<x<8,-8<y<8$}
\end{center}
\end{figure}

\begin{figure}[!htb]
\begin{center}
{
\includegraphics[width=0.8 \columnwidth]{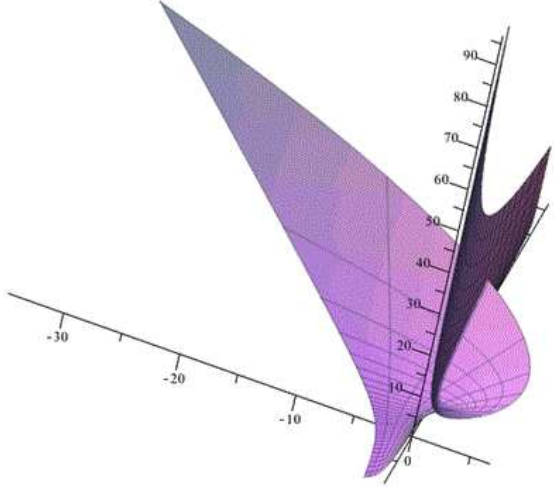}
}
\caption{$0<x<1.5,-0.8<y<3$}
\end{center}
\end{figure}

\rem By computation, we find out that when $\Re (b) = -\frac 12$ the simple element with 3 simple poles will give the one soliton type solution, and other cases it gives two soliton type solution.

 (2) Choose simple element
 $Ag_{\a,\ell_1}(\l)g_{\bar \a^{-1},\ell_2}(\l)$, where $\ell_{2}=\C(b_2,c_2,1)$, $\a \in \C^{\times}/S^1$ satisfy $\Psi_{\a,b_2,c_2}>0$ (see Remark \ref{r45}), and $E(z,\bar z,\l)$ is \eqref{6a1}. Then we have the following factorization:
 \begin{align*}
Ag_{\a,\ell_1}(\l)g_{\bar \a^{-1},\ell_2}(\l)E(\l)=\tilde{E}(\l)\tilde A g_{\a,\tilde \ell_1}(\l)g_{\bar \a^{-1},\tilde \ell_2}(\l),
\end{align*}
where
\[
\tilde l_2=l_2 E(\bar\a^{-1})=\C(\tilde b_2,\tilde c_2, 1),
\]
\begin{align} \label{6jm}
\bpm \tilde b_2\\ \tilde c_2\\ 1 \epm =\bpm \frac{(b_2+c_2+1)R(\bar\a^{-1})+(b_2+c_2\epsilon^{4}+\epsilon^{2})R(\epsilon^{2}\bar\a^{-1})+(b_2+c_2\epsilon^{2}+\epsilon^{4})R(\epsilon^{4}\bar\a^{-1})}{(b_2+c_2+1)R(\bar\a^{-1})+(b_2\epsilon^{4}+c_2\epsilon^{2}+1)R(\epsilon^{2}\bar\a^{-1})+(b_2\epsilon^{2}+c_2\epsilon^{4}+1)R(\epsilon^{4}\bar\a^{-1})}\\
\frac{(b_2+c_2+1)R(\bar\a^{-1})+(b_2\epsilon^{2}+c_2+\epsilon^{4})R(\epsilon^{2}\bar\a^{-1})+(b_2\epsilon^{4}+c_2+\epsilon^{2})R(\epsilon^{4}\bar\a^{-1})}{(b_2+c_2+1)R(\bar\a^{-1})+(b_2\epsilon^{4}+c_2\epsilon^{2}+1)R(\epsilon^{2}\bar\a^{-1})+(b_2\epsilon^{2}+c_2\epsilon^{4}+1)R(\epsilon^{4}\bar\a^{-1})}\\ 1 \epm ,
\end{align}
\[
\tilde A=diag(\tilde d,\tilde d^{-1},1)
\]
\[
\tilde d=\sqrt{(2\tilde{b}_1\tilde{c}_1-1)(2\tilde{b}_2\tilde{c}_2-1)}=\sqrt{\frac{\Psi_{\a,\tilde b_2,\tilde c_2}}{(-|\a|^4|\tilde b_2|^2 +|\a|^2 |\tilde c_2|^2 -\frac12  (1-|\a|^6))^2}}.
\]
 By Theorem \ref{t55}, we can claim that  $h = e^{\tilde{\psi}} = d^2$  is the new solution of Tzitz\'{e}ica equation \eqref{eqgc}.
 One can check this by Maple.

In particular, we choose $c_1=\frac{1}{2}+\frac{\sqrt{3} i}{2}, c_{2}=0,\a=-\frac{1}{2}i, d=1$. Then we can get the $\tau$-function of the new solution $h$:
\begin{align*}
\tau=&1+(\frac{45}{26}-\frac{55 \sqrt{3} i}{78}) e^{{\frac {15}{4}}y-\frac{5}{4}\sqrt {3}x+\frac{9}{4}ix+\frac{3}{4}\sqrt {3}iy}
+(\frac{45}{26}+\frac{55 \sqrt{3} i}{78}) e^{{\frac {15}{4}}y-\frac{5}{4}\sqrt {3}x-\frac{9}{4}ix-\frac{3}{4}\sqrt {3}iy} \\
&+e^{\frac{15}{2}y-\frac{5}{2}\sqrt {3}x}.
\end{align*}
It is coincide with the $\tau$-functions of 2-soliton solutions \eqref{6d} with parameters $k_1=-\frac{\sqrt{3}}{4}-\frac{3 i}{4}$, $k_2=-\sqrt{3}-3i$, $s_1=ln(\frac{135 - 55\sqrt{3} i}{78})$, $s_2=ln(\frac{135 + 55\sqrt{3} i}{78})$.
\begin{figure}[!htb]
\begin{center}
{
\includegraphics[width=0.45 \columnwidth]{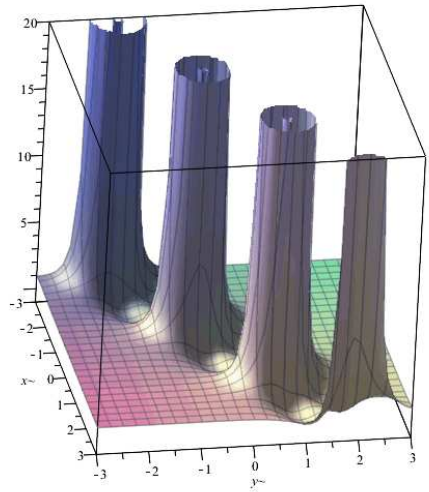}
}
\caption{$-3<x<3,-3<y<3,-1<z<20$}
\end{center}
\end{figure}

At last, we give the picture of this solution and the new definite affine sphere in Figure 5 and Figure 6 respectively.

\rem In particularly,  we choose $\ell = \mathbb{C} \cdot (1, 1, 1)$ and any $\a \in \mathbb{C}^{\times}/S^1$. From \eqref{6jm}, we can see that $\tilde{\ell}_1 = \tilde{\ell}_2 = \mathbb{C} \cdot (1, 1, 1)$, which means the dressing actions is trivial, and this special simple element is commutative with the special affine frames in this case.

\begin{figure}[!htb]
\begin{center}
{
\includegraphics[width=0.5 \columnwidth]{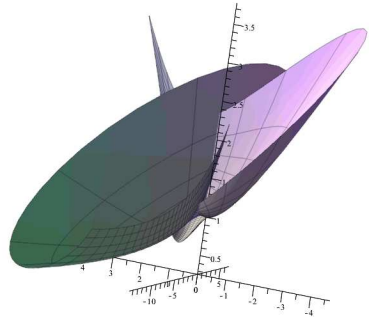}
}
\caption{$-0.4<x<0.8,-1.5<y<1.5$}
\end{center}
\end{figure}

\vfill
\newpage

\section*{Acknowledgments}

The authors would like to thank  Chuu-Lian Terng, Bo Dai and Zi-Xiang Zhou for the helpful discussions. The third author would like to express the deepest gratitude for the support of the Hong Kong University of Science and Technology during the project, especially Min Yan, Yong-Chang Zhu, Bei-Fang Chen and Guo-Wu Meng. The authors had also been supported by the NSF of China (Grant Nos. 10941002, 11001262), and the Starting Fund for Distinguished Young Scholars of Wuhan Institute of Physics and Mathematics (Grant No. O9S6031001).

\end{document}